\documentclass[draftcls, onecolumn]{IEEEtran}
\usepackage{amsfonts}
\usepackage{epsfig}
\usepackage{amsmath, amssymb, amsfonts, mathrsfs, amsthm, color, calrsfs}
\usepackage{enumerate}

\ifCLASSINFOpdf
\else
\fi

\newtheorem{thm}{Theorem}
\newtheorem{deff}[thm]{Definition}

\newtheorem{cor}[thm]{Corollary}

\newcommand{\ee}{\varepsilon}

\newcommand{\R}{\mathbb{R}}

\newcommand{\E}{\mathcal{E}}

\newcommand{\cD}{\mathcal{D}}

\newcommand{\cM}{\mathcal{F}}
\newcommand{\cF}{\mathcal{M}}
\newcommand{\Orliczdual}{\widetilde{\Omega}_{\phi}^{orlicz}}
\newcommand{\OrliczdualG}{\widetilde{G}_{\phi}^{orlicz}}

\begin{document}
%
\title{Orlicz addition for measures and an optimization problem for the $f$-divergence}
%
%
%

\author{Shaoxiong~Hou and
Deping~Ye
}

\maketitle

\begin{abstract}
 In this paper, the Orlicz addition of measures is proposed and an interpretation of the $f$-divergence is provided based on a linear Orlicz addition of two measures. Fundamental inequalities, such as, a dual functional Orlicz-Brunn-Minkowski inequality,  are established. We also investigate an optimization problem for the $f$-divergence and establish functional affine isoperimetric inequalities for the dual  functional  Orlicz affine and geominimal surface areas of measures.

\end{abstract}

\begin{IEEEkeywords}
affine isoperimetric inequality, affine surface area, Brunn-Minkowski theory, dual Brunn-Minkowski theory, the $f$-divergence, geominimal surface area, optimization problem for the $f$-divergence.
\end{IEEEkeywords}

%
\IEEEpeerreviewmaketitle

\section{introduction}

Let $\Omega$ be a nonempty set and $\mu$ be a measure on $\Omega$. Assume that $P$ and $Q$ are two finite measures on $\Omega$ whose  
density functions $p$ and $q$, respectively, with respect to $\mu$ are positive on $\Omega$. That is, $p, q>0$ such that $$P(\Omega)=\int_\Omega
  p\,d\mu<\infty \ \ \ \mathrm{and} \ \ \ Q(\Omega)=\int_\Omega
  q\,d\mu<\infty. $$ For a real valued function $f$, the  $f$-divergence of $P$ and $Q$,  denoted by $D_f(P,Q)$, was  introduced independently by Ali and Silvey \cite{as66}, Csisz\'{a}r \cite{c63} and Morimoto \cite{m63}. It can be formulated by \begin{equation}\label{2016-1-20-f-def}
D_f(P,Q)=
\int_{\Omega}f\left(\frac{p}
{q}\right)q\,d\mu.
\end{equation}
The $f$-divergence is an extension of the classical $L_p$ distance of measures and contains many widely-used distances for measures as its special cases,  e.g., Bhattcharyya distance,  Kullback-Leibler divergence, Renyi distance, $\chi^2$-distance  and total variation distance (see e.g. \cite{al61, bh43, fn14, fv06, kl51}). Moreover,  if $f$ is strictly convex with $f(1)=0$ and $P(\Omega)=Q(\Omega)\neq 0$, Jensen's inequality implies that,
 \begin{equation*}  D_f(P, Q)  =  \int_\Omega
  f\! \left(\!\frac{p}{q}\!\right)q\,d\mu\geq f\left(\!\frac{P(\Omega)}{Q(\Omega)}\!\right)\cdot Q(\Omega)=0,\end{equation*} with equality if and only if $p=q$ almost everywhere with respect to $Q$. When $f$ is strictly concave with $f(1)=0$, one gets similar results with ``$\geq$'' replaced by ``$\leq$". From this viewpoint, the $f$-divergence can be used to distinguish two measures. Without doubt, the $f$-divergence plays fundamental roles in, such as, image analysis,  information theory,  pattern matching  and statistical learning  (see \cite{bgv90, ct06, ht01, lv06, ov03}), where the measure of difference between measures is required. Moreover, in general, the $f$-divergence is arguably better than the $L_p$ distance. 

 Recent development in convex geometry has witnessed the strong connections between the $f$-divergence and convex geometry. For instance, it has been proved that the  $L_p$ affine surface area  \cite{Blas, Lu1, SWerner}, a central notion in convex geometry, is related to the Renyi entropy \cite{w12}; while the general affine surface area \cite{Ludwig2009, LR1} is associated to the $f$-divergence \cite{CaglarWerner}.  Note that these affine surface areas are valuations; and valuations are the key ingredients for the Dehn's solution of Hilbert third problem. Moreover, under certain conditions (such as, semicontinuity), it has been proved that these affine surface areas can be used to uniquely  characterize all valuations which remain unchanged under linear transforms with determinant $\pm 1$ (see e.g.  \cite{HaberlParapatits2014, LudR, LR1}). On the other hand, as showed in the Subsection \ref{optimal-geominimal}, the affine and geominimal surface areas (see e.g. \cite{Lu1, Petty74, y15,Ye2014a1})  can be translated to an optimization problem  for the $f$-divergence. This observation leads us to investigate the dual functional Orlicz affine and geominimal surface areas for measures, which are invariant  under linear transforms with determinant $\pm 1$.

The Brunn-Minkowsi inequality is arguably one of the most important inequalities in convex geometry.  It can be used to prove, for instance, the celebrated Minkowski's and isoperimetric inequalities. (Note that the isoperimetric problem has a history over 1000 years).  See the excellent survey \cite{gardner-b-m} for more details.  On the other hand, the dual Brunn-Minkowski inequality and dual Minkowski inequality are crucial for the solutions of the famous Busemann-Petty problem (see e.g., \cite{Gardner1994, GardnerKoldobski1999, Lut1988, Zhang1999}). The Brunn-Minkowsi inequality and its dual have been extended to the Orlicz theory in  \cite{ghw14, ghwy14, xjl14, zzx14}.

This paper is dedicated to provide a basic theory for the dual functional Orlicz-Brunn-Minkowski theory of measures and establish an interpertation for the $f$-divergence. In particular, we define the Orlicz addition of measures and prove the dual functional Orlicz-Brunn-Minkowski inequality. Moreover, we show that the $f$-divergence is the first order variation of the total mass of a measure obtained by a linear Orlicz addition of two measures. Further connections between the $f$-divergence and (convex) geometry are provided.  We also investigate an optimization problem for the $f$-divergence, and define the dual functional Orlicz affine and geominimal surface areas for measures. Related functional affine isoperimetric inequalities for the dual  functional  Orlicz affine and geominimal surface areas for measures are established.

\section{Orlicz addition for measures}
In this section, we define the Orlicz addition for multiple measures and discuss basic properties for the operation.  
 
\subsection{Orlicz addition for functions: definition and properties}

Throughout this paper, $\Omega$ is a nonempty set and $m\geq 1$ is an integer. Denote by $\cM$ the set of nonnegative real-valued measurable functions defined on $\Omega$.  We use $\cM^{+}$ to denote the set of all functions in $\cM$ which are positive, and $\cM^{+c}$  for the set of all functions in $\cM^+$ which are also continuous.

 Let $\Phi_m$ denote the set of all continuous functions $\varphi:[0,\infty)^m\to [0,\infty)$ that are strictly increasing in each component with $\varphi(o)= 0$ and $\lim_{t\to \infty} \varphi(tz) =\infty$ for each nonzero $z\in [0,\infty)^m$.  Hereafter $o=(0, \cdots, 0)$ stands for the origin of $\R^m$.
 Let $\Psi_m$ denote the set of all continuous functions $\varphi: (0,\infty)^m\to(0,\infty)$ that are strictly decreasing in each component with $\lim_{t\to 0} \varphi(tz) =\infty$ and $\lim_{t\to \infty} \varphi(tz) =0$ for each $z\in (0,\infty)^m$.   Note that $\varphi(x)=x_1^p+\cdots+x_m^p$ belongs to $\Phi_m$ if $p>0$ and belongs to $\Psi_m$ if $p<0$.

 The Orlicz addition of functions is defined as follows.
  \begin{deff}\label{orliczadditionmeasre-1}   For $\!\varphi \in \! \Phi_m$,   $\widetilde{+}_{\varphi}(p_1,\cdots, p_m)$,  the Orlicz addition of  functions  $p_1, \cdots\!, p_m\! \in\! \cM$,  is (uniquely and implicitly) defined by   \begin{equation}\label{Orldef-function}
\varphi\!\left(\! \frac{p_1(x)}{\widetilde{+}_{\varphi}
(p_1,\cdots\!,p_m)(x)},\cdots\!, \frac{p_m(x)}
{\widetilde{+}_{\varphi}(p_1,\cdots\!,p_m)(x)}\!\right) = 1,
\end{equation}
if  $p_1(x)+\cdots+p_m(x)>0$, and  otherwise by
$$\widetilde{+}_{\varphi}(p_1,\cdots,p_m)(x)=0.$$
If $\varphi\in \Psi_m$ and in addition $p_1, \cdots, p_m \in \cM^+$, the Orlicz addition  $\widetilde{+}_{\varphi}(p_1,\cdots, p_m)$ is defined by equation (\ref{Orldef-function}).
\end{deff}

 \noindent{\bf Remark.} Although in Definition \ref{orliczadditionmeasre-1}, the functions $p_1, \cdots, p_m$ are assumed to be measurable, equation (\ref{Orldef-function}) can also be used to define the Orlicz addition of general nonnegative functions. In an independent work \cite{gk2016}, Gardner and Kiderlen also provided the definition for the Orlicz addition of nonnegative functions. The second author of this paper would like to thank Professor Gardner for mentioning  \cite{gk2016} to him. It is worth to mention that the major concentrations of these two papers are completely different:  this paper focuses on the Orlicz addition of measures, an interpretation of the $f$-divergence and related inequalities; while the paper \cite{gk2016} mainly aims to provide a structural theory of operations between real-valued functions.

  Clearly if $\varphi\in {\Phi}_m$, then $\widetilde{+}_{\varphi}(p_1,\cdots, p_m)(x)=0$ implies that $p_1(x)=\cdots=p_m(x)=0$. Moreover,  $\widetilde{+}_{\varphi}(p_1,\cdots, p_m) \in \cM$ if $p_1,\cdots, p_m\in \cM$.  In later context, when $\varphi \in \Psi_m$, the functions $p_1, \cdots, p_m$ in $\widetilde{+}_{\varphi}(p_1,\cdots, p_m)$  are always assumed to be in $\cM^+$.

  It is worth to mention that if $\varphi\in \Phi_m$, $\widetilde{+}_{\varphi}(p_1,\cdots,p_m)(x)$ for $x\in \Omega$ given in Definition \ref{orliczadditionmeasre-1} is equal to the infimum of $\Lambda(x)\subset \R$, where  \begin{eqnarray*}\Lambda(x)=
 \left\{\lambda>0: \varphi\left(\frac{p_1(x)}{\lambda}
,\cdots,\frac{p_m(x)}{\lambda}\right)\le 1\right\}.
\end{eqnarray*} If $\varphi\in \Psi_m$ and $p_i\in \cM^+$ for all $i=1, \cdots, m$, then  $\widetilde{+}_{\varphi}(p_1,\cdots,p_m)(x)$ for $x\in \Omega$ is equal to the supremum of $\Lambda(x)\subset \R$. To this end, if $\varphi\in \Phi_m$ and $p_1(x)=\cdots=p_m(x)=0$, then $\Lambda(x)=\{\lambda: \lambda>0\}$ and hence $\inf \Lambda(x)=0$ as desired. Now assume that $\sum_{j=1}^mp_j(x)>0$ which yields $$(p_1(x), \cdots, p_m(x))\neq o.$$  It is easy to see that  $$\widetilde{+}_{\varphi}(p_1,\cdots,p_m)(x)\in \Lambda(x)$$ by formula (\ref{Orldef-function}).  On the other hand,  the fact that $\lim_{t\to \infty} \varphi(tz) =\infty$ for each nonzero $z\in [0,\infty)^m$  implies  $\widetilde{+}_{\varphi}(p_1,\cdots,p_m)(x)>0$.  Formula (\ref{Orldef-function}) together with the fact that $\varphi$ is strictly increasing in each component  imply that for all  $0<\lambda< \widetilde{+}_{\varphi}(p_1,\cdots, p_m)(x)$,  \begin{eqnarray*}
&&\varphi\left( \frac{p_1(x)}{\lambda}, \cdots, \frac{p_m(x)}{\lambda}\right)
> 1.
\end{eqnarray*} Thus, $ \Lambda(x)=\big[\widetilde{+}_{\varphi}(p_1,\cdots,p_m)(x), \infty\big)$ and $$\widetilde{+}_{\varphi}(p_1,\cdots,p_m)(x)=\inf \Lambda(x).$$ Along the same lines, one can get the desired argument for the case $\varphi\in \Psi_m$.

Now we prove the basic properties of $\widetilde{+}_{\varphi}(p_1,\cdots, p_m)(x)$  where $p_1, \cdots, p_m \in \cM$.

\begin{thm}\label{dualOrthm1}
Let $m\ge 2$ and $\varphi\in \Phi_m$.

\vskip 2mm
\noindent{\rm(i)} For $r\ge 0$, one has, $$\widetilde{+}_{\varphi}(rp_1,\cdots, rp_m)=r\cdot \widetilde{+}_{\varphi}(p_1,\cdots, p_m).$$

\noindent{\rm(ii)} Assume that  $\varphi\in {\Phi}_m$ satisfies $\varphi(e_j)=1$ for all $j=1, \cdots, m$, where $\{e_1, \cdots, e_m\}$ is the standard orthonormal basis of $\R ^m$. Then, for $j=1,\cdots,m$, one has  $$\widetilde{+}_{\varphi}(0,\cdots,0, p_j, 0, \cdots, 0)=p_j. $$

\noindent{\rm(iii)} If $q_1, \cdots, q_m\in \cM$ such that $p_j\leq q_j$ for  all $j=1, \cdots, m$, then $$\widetilde{+}_{\varphi}(p_1,\cdots, p_m)\leq \widetilde{+}_{\varphi}(q_1,\cdots, q_m).$$ In particular,
\begin{equation}\label{formulawithequal}  {\widetilde{+}_{\varphi}(p_1,\cdots\!,p_m)}\leq {\tau_0}^{-1}  \cdot {\sum_{j=1}^m p_j}, \end{equation} where  $\tau_0>0$ satisfies $\varphi(\tau_0, \cdots, \tau_0)=1$.

\vskip 2mm
\noindent{\rm(iv)} Assume that $p_{ij}\in \cM$ for $j=1, \cdots, m$ and $i=1, 2, \cdots,$ such that, for all $x\in \Omega$ and for all $j=1, \cdots, m$,
$$\lim_{i\rightarrow \infty} p_{ij}(x)= p_j(x).$$
Then, for all $x\in \Omega$,
$$\lim_{i\rightarrow \infty} \widetilde{+}_{\varphi}(p_{i1},\cdots, p_{im})(x)= \widetilde{+}_{\varphi}(p_1,\cdots, p_m)(x).$$

\noindent {\rm (v)} Let $p_{ij}$ be as in (iv) and let $S\subset \Omega$ be a compact set. Assume that all functions $p_{ij}$ are positive and continuous on $S$,  and the sequence $p_{ij}$ is uniformly convergent to $p_j\in \cM^+$ on $S$ as $i\rightarrow \infty$. Then $\widetilde{+}_{\varphi}(p_{i1},\cdots, p_{im})$ is convergent to $\widetilde{+}_{\varphi}(p_1,\cdots, p_m)$ uniformly on $S$ as $i\rightarrow \infty$.

The above statements except statement (ii) still hold true when $\varphi\in \Psi_m$ and all functions involved are positive, except that $r\geq 0$ should be replaced by $r>0$ in (i).
\end{thm}

\begin{proof}  We only prove the results for $\varphi\in \Phi_m$, and the case $\varphi\in \Psi_m$ follows along the same lines.

\vskip2mm \noindent{\rm(i)} The equality holds trivially if $r=0$ or $\sum_{j=1}^m p_j(x)=0$.  Let $x\in \Omega$ be such that $\sum_{j=1}^m p_j(x)>0$. The desired equality for $r>0$ follows from the fact that for all $(a_1, \cdots, a_m)\neq o$, the equation $$\varphi\bigg(\frac{a_1}{\lambda}, \cdots, \frac{a_m}{\lambda}\bigg)=1$$ has a unique solution  and the fact that 
\begin{eqnarray*}1\!\!&=&\!\! \varphi\!\left(\! \frac{rp_1(x)}{\widetilde{+}_{\varphi}
(rp_1,\cdots\!, rp_m)(x)},\cdots\!, \frac{rp_m(x)}
{\widetilde{+}_{\varphi}(rp_1,\cdots\!, rp_m)(x)}\!\right)\\ \!\! &=&\!\! \varphi\!\left(\! \frac{p_1(x)}{\left(\frac{\widetilde{+}_{\varphi}
(rp_1,\cdots\!, rp_m)(x)}{r}\right)},\cdots\!, \frac{p_m(x)}{\left(\frac{\widetilde{+}_{\varphi}
(rp_1,\cdots\!, rp_m)(x)}{r}\right)}\!\right).
\end{eqnarray*}

\noindent{\rm(ii)} If $p_j(x)=0$, then $\widetilde{+}_{\varphi}(0, \cdots, 0, p_j, 0, \cdots, 0)(x)=0$. Assume that $p_j(x)\neq 0$. Formula (\ref{Orldef-function}) implies that  $$\varphi\left( 0,\cdots, 0, \frac{p_j(x)}{\widetilde{+}_{\varphi}(0,\cdots, 0, p_j, 0, \cdots, 0)(x)}, 0, \cdots, 0\right) = 1.$$ Together with the facts that $\varphi(e_j)=1$ and $\varphi$ is strictly increasing in each component, one gets  $$\widetilde{+}_{\varphi}(0,\cdots,0,p_j,0,\cdots,0)(x)=p_j(x).$$

\noindent{\rm(iii)} The desired result holds trivially if $\sum_{j=1}^m p_j(x)=0$. Assume that $p_j\leq q_j$ for all $j=1, \cdots, m$ and  $\sum_{j=1}^m p_j(x)>0$. Note that $\varphi$ is strictly increasing in each component. By equation (\ref{Orldef-function}), one has,  
 \begin{eqnarray*}1\!\!&=&\!\!\varphi\!\left(\! \frac{q_1(x)}{\widetilde{+}_{\varphi}
(q_1,\cdots\!, q_m)(x)},\cdots\!, \frac{q_m(x)}
{\widetilde{+}_{\varphi}(q_1,\cdots\!, q_m)(x)}\!\right) \\ \!\!&=&\!\!\varphi\!\left(\! \frac{p_1(x)}{\widetilde{+}_{\varphi}
(p_1,\cdots\!,p_m)(x)},\cdots\!, \frac{p_m(x)}
{\widetilde{+}_{\varphi}(p_1,\cdots\!,p_m)(x)}\!\right) \\ \!\!&\leq&\!\! \varphi\!\left(\! \frac{q_1(x)}{\widetilde{+}_{\varphi}
(p_1,\cdots\!,p_m)(x)},\cdots\!, \frac{q_m(x)}
{\widetilde{+}_{\varphi}(p_1,\cdots\!,p_m)(x)}\!\right).
\end{eqnarray*} Again by the fact that $\varphi$ is strictly increasing in each component, one gets  $$\widetilde{+}_{\varphi}(p_1,\cdots, p_m)\leq \widetilde{+}_{\varphi}(q_1,\cdots, q_m).$$

In particular, let $q_1=\cdots=q_m=\sum_{j=1}^{m} p_j$, then
$$\widetilde{+}_{\varphi}(p_1,\cdots, p_m)\leq \widetilde{+}_{\varphi}\bigg(\sum_{j=1}^{m} p_j,\cdots, \sum_{j=1}^{m} p_j \bigg).$$ The right hand side is equal to ${\tau_0}^{-1}  \cdot {\sum_{j=1}^m p_j}$ which follows directly from $\varphi(\tau_0, \cdots, \tau_0)=1$.

\vskip 2mm
 \noindent{\textcolor{red}\rm(iv)}  Assume that  $\sum_{j=1}^m p_j(x)=0$. As $$\lim_{i\rightarrow \infty} p_{ij}(x)=p_j(x),$$ then for all $\epsilon>0$, there is $i(\epsilon)\in \mathbb N$, such that for $i>i(\epsilon)$, $$ \sum_{j=1}^m p_{ij}(x)<\epsilon.$$
By formula (\ref{formulawithequal}), one has, for all $i>i(\epsilon)$, \begin{eqnarray*} 0&\leq&
{\widetilde{+}_{\varphi}(p_{i1},\cdots\!,p_{im})(x)}\\  &\leq& {\tau_0}^{-1}  \cdot {\sum_{j=1}^m p_{ij}(x)} <  {\tau_0}^{-1}  \epsilon.
\end{eqnarray*} Consequently, one has,
$$\lim_{i\rightarrow \infty} {\widetilde{+}_{\varphi}(p_{i1},\cdots\!,p_{im})(x)}=0=\widetilde{+}_{\varphi}
(p_1,\cdots\!,p_m)(x)$$ because $\sum_{j=1}^m p_j(x)=0$.

Now assume that $\sum_{j=1}^m p_j(x)>0$ and $\lim_{i\rightarrow \infty} p_{ij}(x)= p_j(x)$. Then, $$\lim_{i\rightarrow \infty} \sum_{j=1}^m p_{ij}(x)= \sum_{j=1}^{m} p_j(x)>0.$$ Then there is $i_0\in \mathbb N$, such that, $\sum_{j=1}^m p_{ij}(x)>0$  for all $i>i_0$ and hence
 \begin{eqnarray*}1\!\!&=&\!\! \varphi\!\left(\! \frac{p_{i1}(x)}{\widetilde{+}_{\varphi}
(p_{i1},\cdots\!, p_{im})(x)},\cdots\!, \frac{p_{im}(x)}
{\widetilde{+}_{\varphi}(p_{i1},\cdots\!, p_{im})(x)}\!\right).
\end{eqnarray*} Taking the limit as $i\rightarrow \infty$, the desired conclusion follows from the continuity of $\varphi$ and the uniqueness of the solution of \eqref{Orldef-function}. That is,  for $x\in \Omega$ such that $\sum_{j=1}^m p_j(x)>0$, $$\lim_{i\rightarrow \infty} \widetilde{+}_{\varphi}(p_{i1},\cdots, p_{im})(x)= \widetilde{+}_{\varphi}(p_1,\cdots, p_m)(x).$$

\noindent{\rm(v)} Assume that all functions $p_{ij}$ are positive and continuous on $S$, and the sequence $p_{ij}$ is uniformly convergent to $p_j\in \cM^+$ on $S$. Then, there exist $c_1, c_2>0$ such that for all $x\in S$, $c_1\leq p_j(x)\leq c_2$ and $c_1\le p_{ij}(x)\le c_2$ for all $i$ and $j$. Part (iv) implies that  $\widetilde{+}_{\varphi}(p_{i1},\cdots, p_{im})$ converges to $\widetilde{+}_{\varphi}(p_{1},\cdots, p_{m})$ pointwisely on $S$.

 If the convergence is not uniform on $S$, then there exist $\ee_0>0$ and $n_i>i$, such that,  $x_{n_i}\in S$ with $x_{n_i}\rightarrow x_0$ (due to the compactness of $S$), and
\begin{equation}\label{dy1}
\left|\widetilde{+}_{\varphi}(p_{n_i1},\!\cdots\!, p_{n_im})(x_{n_i}\!)-
\widetilde{+}_{\varphi}(p_{1},\!\cdots\!, p_{m})(x_{n_i}\!)\right|\ge \ee_0.
\end{equation}
Part (iii) and the fact $\varphi(\tau_0, \cdots, \tau_0)=1$ imply
\begin{equation*}
c_1/\tau_0 \le \widetilde{+}_{\varphi}(p_{n_i1},\cdots, p_{n_im})(x)\le c_2/\tau_0,
\end{equation*} That is,
$\{\widetilde{+}_{\varphi}(p_{n_i1},\cdots,p_{n_im})(x_{n_i})\}_{i\in \mathbb{N}}$ is a
bounded sequence and hence has a convergent subsequence. Without loss of generality, assume that
$$\lim_{i\rightarrow\infty}  \widetilde{+}_{\varphi}(p_{n_i1},\cdots,p_{n_im})(x_{n_i})=c_0>0,$$ where  $c_0>0$ is a constant. This together with $x_{n_i}\rightarrow x_0$ and $p_{n_ij}\rightarrow p_j$ uniformly on $S$ further imply that
\begin{eqnarray*} 1\!\!\! &=&\!\!\!
\varphi\!\left(\!\frac{p_{n_i1}(x_{n_i})}{\widetilde{+}_{\varphi}
(p_{n_i1},\cdots\!, p_{n_im})(x_{n_i}\!)},\cdots\!,\frac{p_{n_im}(x_{n_i})}
{\widetilde{+}_{\varphi}(p_{n_i1},\cdots\!, p_{n_im})(x_{n_i}\!)}\!\right)
 \\ &\!\!\!\rightarrow&\!\!\!
\varphi\left(\frac{p_{1}(x_{0})}{c_0},\cdots,\frac{p_{m}(x_0)}
{c_0}\right), \ \ \mathrm{as} \  i\rightarrow \infty.
\end{eqnarray*}
It follows that $c_0=\widetilde{+}_{\varphi}(p_{1},\cdots,p_{m})(x_{0})$, which leads to a contradiction with \begin{eqnarray*}
 |c_0-\widetilde{+}_{\varphi}(p_{1},\cdots, p_{m})(x_{0})|\ge\ee_0,
\end{eqnarray*} after taking $i\rightarrow \infty$ from both sides of (\ref{dy1}).
Hence, the desired uniform convergence follows. \end{proof}

\subsection{Orlicz addition for star bodies and convex bodies}

Our definition of the Orlicz addition for measures is motivated by the recently introduced Orlicz additions for  convex bodies and star bodies, which are the foundation of the newly initiated Orlicz-Brunn-Minkowski theory for convex bodies and its dual theory \cite{ghw14, ghwy14, xjl14, zzx14}.

In this subsection, we briefly discuss these Orlicz additions in geometry and show how it can be linked with our Orlicz addition for functions. Notations and concepts for geometry below are standard, and more details can be found in \cite{s14}.

Denote by $S^{n-1}=\{(u_1, \cdots, u_n)\in \R^n: \sum_{i=1}^n u_i^2=1\}$  the unit sphere of $\R^n$.  Throughout this paper, a subset $K\subset \R ^n$ is star-shaped if, for all $x\in K$, the line segment from the origin $o$ to $x$ is contained in $K$. The radial
function of  a star-shaped subset $K$, $\rho_K: S^{n-1}\rightarrow \R$, is defined by
$$
\rho_{K}(u)= \max\{c\ge 0:cu\in K\}, \ \ \  u\in S^{n-1}.
$$ The radial function can be extended to $\R ^n\setminus \{o\}$ by $\rho_K(ru)=r^{-1}\rho_K(u)$ for all $r>0$ and $u\in S^{n-1}$. Note that such an extension is of homogenous degree $-1$.  A star body $K$ is a star-shaped subset with continuous and positive radial function $\rho_K$. Clearly, a star body $K$ is compact with $o$ in its interior.

The radial Orlicz sum of star bodies $K_1, \cdots, K_m$, denoted by $\widetilde{+}_{\varphi}(K_1,\dots,K_m)$, is determined by its radial function $\rho_{\widetilde{+}_{\varphi}(K_1,\dots,K_m)}$, the unique solution of the following equation \cite{ghwy14}: for $u\in S^{n-1}$, \begin{equation}\label{Orldef}
\varphi\left( \frac{\rho_{K_1}(u)}{\rho_{\widetilde{+}_{\varphi}(K_1,\dots,K_m)}(u)},\dots, \frac{\rho_{K_m}(u)}{\rho_{\widetilde{+}_{\varphi}(K_1,\dots,K_m)}(u)}\right) = 1.
\end{equation}   Formula (\ref{Orldef}) can also be used to define the radial Orlicz sum of star bodies $K_1, \cdots, K_m$ for $\varphi\in \Psi_m$.  In fact, the radial Orlicz sum of star bodies $K_1, \cdots, K_m$ defined by formula  (\ref{Orldef}) is a special case of the Orlicz addition of functions defined by formula (\ref{Orldef-function}); it can be obtained by  letting $\Omega=S^{n-1}$ and  $p_j=\rho_{K_j}$.  Alternatively, it can be  also obtained by  letting $\Omega=\R^n\setminus \{o\}$ and $p_j(x)=r^{-1} \rho_{K_j}(u)$ for all $x=ru\neq o$.

The set $K\subset \R^n $ is said to be a convex body if $K$ is a star body such that for all $x, y\in K$, the line segment from $x$ to $y$ is contained in $K$. An arguably better way to characterize convex body $K$ is its support function $h_K: S^{n-1}\rightarrow \{t: t\geq 0\}$, which can be defined by: $$h_K(u)=\max_{x\in K} \langle x, u\rangle, \ \ \ \mathrm{for}\ u\in S^{n-1}.$$ Clearly, $h_K$ is a sublinear function. Let $\Omega=S^{n-1}$, $p_j=h_{K_j}$, and let  $\varphi\in \Phi_m$  be convex,  then formula (\ref{Orldef-function}) becomes:  for $u\in S^{n-1}$,
 \begin{equation}\label{orlicz-addition-convex--11----111}
\varphi\left(\frac{h_{K_1}(u)}{h_{{+}_{\varphi}(K_1,\dots,K_m)}(u)},\dots, \frac{h_{K_m}(u)}{h_{{+}_{\varphi}(K_1,\dots,K_m)}(u)}\right) = 1.
\end{equation} The unique solution of (\ref{orlicz-addition-convex--11----111}) is exactly the support function of ${+}_{\varphi}(K_1,\dots,K_m)$, the Orlicz addition of convex bodies $K_1, \cdots, K_m$ \cite{ghw14}.

\subsection{Orlicz addition for measures and a dual functional Orlicz-Brunn-Minkowski inequality}
Let $\mu$ be a given measure on $\Omega$ such that $\mu(\Omega)\neq 0$. Denote by $\cF$ the set of finite measures on $\Omega$ that are absolutely continuous with respect to $\mu$ and whose density functions with respect to $\mu$ are in  $\cM$. That is, $P\in \cF$ has the density function $p$ with respect to $\mu$ such that   $p\in \cM$,  $P(\Omega)<\infty$,  and $$P(A)=\int_{A}p(x)\,d\mu(x), \ \ \mathrm{for\  all\ measurable}\ A\subseteq \Omega.$$ In this paper, we always assume that $\cF\neq \emptyset$. Let $\cF^{+}$ and $\cF^{+c}$ denote the sets of all measures in $\cF$ whose density functions are in $\cM^+$ and  in $\cM^{+c}$, respectively. Note that  $\widetilde{+}_{\varphi}(p_1,\cdots, p_m) \in \cM$ and    \begin{eqnarray*} \int _{\Omega} \widetilde{+}_{\varphi}(p_1,\cdots, p_m)(x)\,d\mu(x) \!\!&\leq&\!\!  \int _{\Omega} {\tau_0}^{-1}  \cdot {\sum_{j=1}^m p_j(x)} \,d\mu(x)\\ \!\!&=&\!\!{\tau_0}^{-1}  \cdot {\sum_{j=1}^m  \int _{\Omega}  p_j(x)} \,d\mu(x) \\ \!\!&<&\!\! \infty, \end{eqnarray*} where the first inequality follows from inequality (\ref{formulawithequal}). That is, $\widetilde{+}_{\varphi}(p_1,\cdots, p_m)$ can be the density function of a measure in $\cF$.  This observation leads to our definition for the Orlicz addition of measures.

  \begin{deff}\label{orliczadditionmeasre-2} Let $P_1, \cdots, P_m\in \cF$ with density functions $p_1, \cdots, p_m\in \cM$.  For $\varphi\in \Phi_m$,  the Orlicz addition of measures $P_1, \cdots, P_m$, denoted by $\widetilde{+}_{\varphi}(P_1,\cdots, P_m)$, is the measure in $\cF$ whose density function   is $\widetilde{+}_{\varphi}(p_1,\cdots\!,p_m)$.  Similarly, the Orlicz addition of $P_1, \cdots, P_m\in \cF^+$ for $\varphi\in \Psi_m$   is a measure in $\cF^+$ whose density function   is $\widetilde{+}_{\varphi}(p_1,\cdots\!,p_m)$. \end{deff}

 \noindent{\bf Remark.} Clearly, if  $P_1, \cdots, P_m\in \cF^+$ (or $\cF^{+c}$, respectively), then $\widetilde{+}_{\varphi}(P_1,\cdots, P_m)\in \cF^+$ (or $\cF^{+c}$, respectively).  In later context,  for $\varphi \in \Psi_m$, the measures $P_1, \cdots, P_m$ in $\widetilde{+}_{\varphi}(P_1,\cdots, P_m)$ are always assumed to be in $\cF^+$.

The following theorem provides a dual functional Orlicz-Brunn-Minkowski inequality for the Orlicz addition of measures.
\begin{thm}\label{dualOBMI1} Let $m\ge 2$ and let $P_j\in \cF$ with density functions $p_j$ for $j=1, \cdots, m$.  Assume that  $A\subset \Omega$ is measurable with $\mu(A)\neq 0$ such that $\sum_{j=1}^mp_j(x)>0$ for $x\in A$ almost everywhere with respect to $\mu$. If  $\varphi\in \Phi_m\cup \Psi_m$ is concave, then
\begin{equation}\label{obmi}
\varphi\left(\!\frac{P_1(A)}
{\widetilde{+}_{\varphi}(P_1,\cdots\!, P_m)(A)},\cdots\!,\frac{P_m(A)}
{\widetilde{+}_{\varphi}(P_1,\cdots\!, P_m)(A)}\! \right)\ge 1.
\end{equation} If $\varphi\in \Phi_m\cup \Psi_m$ is convex, the inequality holds with $\ge$ replaced by $\le$.

If $\varphi\in \Phi_m\cup \Psi_m$ is strictly concave or convex,
and $P_1, \cdots, P_m \in \cF^{+c}$,
equality holds if and only if there are constants $a_j>0$ such that  $p_j=a_j p_1$ for $2\leq j\leq m$. \end{thm}

\begin{proof} Let  $\varphi\in \Phi_m$ and $\sum_{j=1}^mp_j(x)>0$ for $x\in A$ almost everywhere with respect to $\mu$. By  inequality (\ref{formulawithequal}), for $x\in A$ almost everywhere with respect to $\mu$, $$0<\widetilde{+}_{\varphi}(p_1,\cdots, p_m)(x)\leq \tau_0^{-1} \sum_{j=1}^m p_j(x).$$ Together with $\mu(A)\neq 0$, one has  $$0<\widetilde{+}_{\varphi}(P_1,\cdots\!, P_m)(A)<\infty.$$  Hence, we can define a  probability measure $\,d\nu$  on $A$ by  $$d \nu=\frac{\widetilde{+}_{\varphi}(p_1,\cdots,
p_m)}{ \widetilde{+}_{\varphi}(P_1,\cdots, P_m)(A)}
\,d\mu.$$ Assume that  $\varphi\in \Phi_m$ is concave.  By (\ref{Orldef-function}) and
Jensen's inequality (see e.g. Proposition 2.2 in \cite{ghwy14}), one has,
\begin{eqnarray*}
1\!\!\! &=&\!\!\!\! \int_{A}\varphi\!\left(\!\frac{p_1(x)}
{\widetilde{+}_{\varphi}(p_1,\cdots\!, p_m)(x)}, \cdots\!,
\frac{p_m(x)}{\widetilde{+}_{\varphi}(p_1,\cdots\!, p_m)(x)}
\!\right)\!\!\,d\nu(x)\\
\!\!\!
&\leq&\!\!\! \varphi\! \left(\int_{A}\frac{p_1(x)}
{\widetilde{+}_{\varphi}(p_1,\cdots, p_m)(x)}\,
d\nu(x),\cdots \right. \\ &&\ \ \ \ \ \ \ \left. \cdots, \int_{A}\frac{p_m(x)}
{\widetilde{+}_{\varphi}(p_1,\cdots, p_m)(x)}\,
d\nu(x)\right)\\
\!\!\!
&=&\!\!\! \varphi\left(\frac{P_1(A)}
{\widetilde{+}_{\varphi}(P_1,\cdots, P_m)(A)},\cdots,
\frac{P_m(A)}
{\widetilde{+}_{\varphi}(P_1,\cdots, P_m)(A)}\right).
\end{eqnarray*} If $\varphi\in \Phi_m$ is a convex function, the above inequality holds with $\leq$ replaced by $\geq$.

Assume that $\varphi$ is strictly concave or strictly convex. Note that
 $P_j\in \cF^{+c}$ has continuous and positive density functions $p_j$ for all $j=1,\cdots,m$. This yields that  $$\frac{p_j}{\widetilde{+}_{\varphi}(p_1,\cdots,
p_m)} \ \ \ \mathrm{for\ all}\ j=1,\cdots, m$$   are positive and continuous on $A$. Hence, equality holds in \eqref{obmi} if and only if there are constants $b_j>0$, such that, for all $x\in A$ and for all $j=1,\cdots, m$,  $$\frac{p_j(x)}{\widetilde{+}_{\varphi}(p_1,\cdots,
p_m)(x)}=b_j. $$ Equivalently, there are constants $a_j>0$ such that  $p_j=a_j p_1$ for $2\leq j\leq m$.

The proof for the case $\varphi\in \Psi_m$ follows along the same lines, and hence is omitted.  \end{proof}

\begin{cor} Let $m\ge 2$ and let $P_j\in \cF$ with density functions $p_j$ for $j=1, \cdots, m$.  Assume that  $A\subset \Omega$ is measurable such that $P_{j_0}(A)>0$ for  some $j_0\leq m$. If $\cF^+\neq \emptyset$ or $\mu(A)<\infty$, then for any  concave function $\varphi$  in $\Phi_m$ such that $\varphi(e_j)=1$ for $j=1, \cdots, m$, one has,
\begin{equation}\label{obmi-cor}
\varphi\left(\!\frac{P_1(A)}
{\widetilde{+}_{\varphi}(P_1,\cdots\!, P_m)(A)},\cdots\!,\frac{P_m(A)}
{\widetilde{+}_{\varphi}(P_1,\cdots\!, P_m)(A)}\! \right)\ge 1,
\end{equation} while the inequality holds with $\ge$ replaced by $\le$ if $\varphi\in \Phi_m$ is convex. \end{cor}

\begin{proof}  Assume that $\mu(A)<\infty$ and $\varphi\in \Phi_m$ such that $\varphi(e_j)=1$ for $j=1, \cdots, m$. Let $p_j\in \cM$ be density functions of $P_j\in \cF$ for $j=1,\cdots,m$. Let $\ee>0$ and $p_j^{\ee}$ be functions defined on $A$ by 
	\begin{equation}\label{obmi-corc}
	p_j^{\ee}(x)=p_j(x)+\ee, \ \ \ \ \mathrm{for} \ \ x\in A.
	\end{equation}
	It is clear that  $p_j^{\ee}\downarrow p_j$
pointwisely on $A$ as $\ee \downarrow 0$. By the arguments of (iii) and  (iv) in Theorem \ref{dualOrthm1},  we get  $$\widetilde{+}_{\varphi}(p_1^{\ee},\cdots,
p_m^{\ee})\downarrow\widetilde{+}_{\varphi}(p_1,\cdots, p_m)$$
pointwisely on $A$  as $\ee \downarrow 0$. The Lebesgue dominated convergence theorem (as $\mu(A)<\infty$) implies that, as $\varepsilon\downarrow 0$,
$$\int_A p_j^{\ee} (x) \,d\mu(x)
\downarrow \int_A p_j (x) \,d\mu(x)$$  for $j=1,\cdots,m$ and
$$
\int_A \widetilde{+}_{\varphi}(p_1^{\ee},\cdots\!,
p_m^{\ee})(x)\,d\mu(x)
\downarrow \int_A \widetilde{+}_{\varphi}(p_1,\cdots\!, p_m)(x)\,d\mu(x).
$$

The statements (ii) and (iii) in Theorem \ref{dualOrthm1}, together with the assumption that $P_{j_0}(A)>0$ for  some $j_0\leq m$, imply
\begin{eqnarray*} \widetilde{+}_{\varphi}(P_1,\cdots\!, P_m)(A) \!\!\!\!&=& \!\!\! \! \int_A \widetilde{+}_{\varphi}(p_1,\cdots, p_m)(x)\,d\mu(x)\\ \!\!\!  &
\geq&\!\!\! \!\!
\int_A \widetilde{+}_{\varphi}(0,\cdots\!,0, p_{j_0}, 0,\cdots\!,0)(x)\,d\mu(x)\\ \! \!\!\! &=&\!\!\! P_{j_0}(A). \end{eqnarray*}
Then,  for all $\ee>0$,  $$\int_A \widetilde{+}_{\varphi}(p_1^{\ee},\cdots,
p_m^{\ee})(x)\,d\mu(x)\geq P_{j_0}(A)> 0.$$  Assume that $\varphi\in \Phi_m$ is concave. By inequality (\ref{obmi}), one has,
\begin{eqnarray}1\!\!&\leq&\!\!
\varphi\left(\!\frac{\int_A p_1^{\ee} (x) \,d\mu(x)}
{\int_A \widetilde{+}_{\varphi}(p_1^{\ee}, \cdots,
p_m^{\ee})(x)\,d\mu(x)},\cdots  \right. \nonumber \\ &&\ \ \ \ \ \ \ \left. \cdots, \frac{\int_A p_m^{\ee} (x) \,d\mu(x)}
{\int_A \widetilde{+}_{\varphi}(p_1^{\ee},\cdots,
p_m^{\ee})(x)\,d\mu(x)}\! \right). \nonumber
\end{eqnarray}  Letting $\ee\downarrow 0$ and by the continuity of $\varphi$, one gets the desired inequality (\ref{obmi-cor}).

 The proof for the other case  $\cF^+\neq \emptyset$ follows along the same lines with $p_j^{\ee}$ in \eqref{obmi-corc} replaced by $$p_j^{\ee}=p_j+\ee p$$  where $p$ is the density function of  any given measure $P\in \cF^+$.  \end{proof}

\subsection{Special cases and applications.}
The above functional Orlicz-Brunn-Minkowski inequalities for the Orlicz addition of measures are important and have many interesting consequences. We will list some of them in both geometry and analysis.

The first one is the following fundamental dual Orlicz-Brunn-Minkowski inequality for star bodies  \cite{ghwy14}. See \cite{zzx14} for a spacial case. For  $\varphi\in \Phi_m\cup \Psi_m$,  let $$\varphi_0(z)=\varphi(z_1^{1/n},\dots, z_m^{1/n})$$ for $z=(z_1, \cdots, z_m)\in [0,\infty)^m$ if $\varphi\in \Phi_m$ and for  $z\in (0,\infty)^m$ if $\varphi\in \Psi_m$.  Let $V_n(K)$ stand for the $n$-dimensional volume of $K$. When $K$ is a star body,
 $$V_n(K)=\frac{1}{n} \int_{S^{n-1}}\rho_K^n(u)\,d\sigma(u),$$ with $\sigma$ the spherical measure on $S^{n-1}$.

\begin{thm}\label{dualOBMI1-geometric}
Let $m, n\ge 2$.  If $\varphi\in \Phi_m\cup \Psi_m$ such that $\varphi_0$ is concave, then for all star bodies $K_1, \cdots, K_m$,
\begin{equation*}
\varphi_0\left(\frac{V_n(K_1)}{V_n(\widetilde{+}_{\varphi}(K_1,\dots,
K_m))},\dots,
\frac{V_n(K_m)}{V_n(\widetilde{+}_{\varphi}(K_1,\dots,
K_m))}\right)\ge 1,
\end{equation*}
while if $\varphi_0$ is convex, the inequality holds with $\geq$ replaced by $\leq$.

If $\varphi_0$ is strictly concave (or convex, as appropriate), equality holds if and only if there exist constants $a_j>0$ such that $\rho_{K_j}=a_j \rho_{K_1}$ for $j=2,\dots,m$.
\end{thm}

In fact, Theorem \ref{dualOBMI1-geometric} follows from Theorem \ref{dualOBMI1} directly by letting $\Omega=S^{n-1}$, $\mu=\sigma$ the spherical measure on $S^{n-1}$, $p_j=\rho_{K_j}^n$,  and by the fact that  $$\big[\rho_{\widetilde{+}_{\varphi}(K_1,\dots,K_m)}\big]^n= \widetilde{+}_{\varphi_0}(\rho_{K_1}^n,\cdots, \rho_{K_m}^n).$$

Let $A\subset \Omega$ be a measurable subset with $\mu(A)\neq 0$. Define  $$\|p\|^s_{s, A}= \int_{A} p(x)^s\,d\mu(x),$$  for $p\in\cM$ if $s>0$ and for $p\in \cM^+$ if $s<0$.   Denote by $\mathcal{L}_{s, A}$ the set of functions with finite $\|\cdot\|^s_{s, A}$, that is, if $p\in \mathcal{L}_{s, A}$, then $\|p\|^s_{s, A}<\infty$.  Let   $\varphi_s(z)=\varphi(z_1^{1/s}, \cdots, z_m^{1/s}).$ 
 We have the following theorem regarding $\|\cdot\|_{s, A}$. 
\begin{thm}\label{min-general-2016}  Let $m\ge 2$ and let $p_j\in \cM^+$ with $0<\|p_j\|_{s, A}<\infty$.    Let $\varphi\in \Phi_m\cup \Psi_m $ such that $\varphi_s$ is concave. Then
\begin{equation*}
\varphi \left(\!\frac{\|p_1\|_{s,A}}
{\|\widetilde{+}_{\varphi}(p_1,\cdots\!, p_m)\|_{s, A}},\cdots\!,\frac{\|p_m\|_{s,A}}
{\|\widetilde{+}_{\varphi}(p_1,\cdots\!, p_m)\|_{s, A}}\! \right)\ge 1.
\end{equation*} If $\varphi_s$ is convex, the inequality holds with $\ge$ replaced by $\le$.

If $\varphi_s$ is strictly concave or convex,
and $p_1, \cdots, p_m \in \cM^{+c}$,
equality holds if and only if there are constants $a_j>0$ such that  $p_j(x)=a_j p_1(x)$ for all $x\in A$ and for $2\leq j\leq m$. \end{thm}

\begin{proof} The desired result follows from Theorem \ref{dualOBMI1} and the following equality: 
\begin{eqnarray*} 1\!\!\!\!\!&=&\!\!\!\!\!
\varphi\!\left(\! \frac{p_1(x)}{\widetilde{+}_{\varphi}
(p_1,\cdots\!,p_m)(x)},\cdots\!, \frac{p_m(x)}
{\widetilde{+}_{\varphi}(p_1,\cdots\!,p_m)(x)}\!\right) \\ 
\!\!\!\!\!&=&\!\!\!\!\! \varphi_s\!\left(\! \frac{[p_1(x)]^s}{[\widetilde{+}_{\varphi}
(p_1,\cdots\!,p_m)(x)]^s},\cdots\!, \frac{[p_m(x)]^s}
{[\widetilde{+}_{\varphi}(p_1,\cdots\!,p_m)(x)]^s}\!\right).
\end{eqnarray*} 
 \end{proof} 
A special case of Theorem \ref{min-general-2016} is the standard Minkowski inequality for the $L_s$ norm of functions with $s\geq 1$. Here, the $L_s$ norm of $g$ is $$\|g\|_{L^s(A)}=\||g|\|_{s, A}.$$ 
In fact, let $m=2$ and $\varphi(x_1, x_2)=x_1 +x_2$, then
 \begin{eqnarray*}
\|g_1+g_2\|_{L^s(A)}&\leq& \||g_1|+|g_2|\|_{s, A}\\
&\leq&
\||g_1|\|_{s, A}+\||g_2|\|_{s, A}\\
&=&\|g_1\|_{L^s(A)}+\|g_2\|_{L^s(A)},
\end{eqnarray*} where the second inequality follows from Theorem \ref{min-general-2016}.

A fundamental object in convex geometry is the $L_s$ mixed volume.  Define $V_s(K, L)$, the $L_s$ mixed volume of  convex bodies $K, L$ with the origin in their interior, by \begin{equation}\label{Ls-mixed-volume-2016} V_s(K, L)=\frac{1}{n} \int_{S^{n-1}} h_L^s(u) h_K^{1-s} (u)\,dS_K(u),\end{equation}  where $S_K$ on $S^{n-1}$ is the surface area measure of $K$. Note that  $V_s(K, K)=V_n(K)$, the volume of $K$. Let $\Omega=S^{n-1}$ and $n\cdot d\mu=h_K^{1-s}\, dS_K$, then $$V_s(K, L)=\|h_L\|^s_{s, S^{n-1}}.$$ Together with Theorem \ref{min-general-2016} and formula (\ref{orlicz-addition-convex--11----111}), one gets the following Orlicz-Brunn-Minkowski type inequality for the $L_s$ mixed volumes, which is new to the literature of geometry.

\begin{thm}\label{min-general-2016-geometry}  Let $m\ge 2$ and let $K, K_1, \cdots, K_m$ be convex bodies with the origin in their interiors.  Let  $\varphi\in \Phi_m\cup \Psi_m $ such that $\varphi_s$ is concave. Then
\begin{eqnarray*} 
\varphi_s \!\left(\!\!\frac{V_s(K, K_1)}
{V_s(K, {+}_{\varphi}(\!K_1,\!\cdots\!, K_m)\!)},\cdots\!, \frac{V_s(K, K_m)}
{V_s(K, {+}_{\varphi}(\!K_1,\!\cdots\!, K_m)\!)}\! \!\right)\! \ge \!1.
\end{eqnarray*} If $\varphi_s$ is convex, the inequality holds with $\ge$ replaced by $\le$.

If $\varphi_s$ is strictly concave or convex, equality holds if and only if there are constants $a_j>0$ such that  $h_{K_j}=a_j h_{K_1}$ for all $2\leq j\leq m$. \end{thm}

 \section{An interpretation of the $f$-divergence}

 A special case of the Orlicz addition of functions $p_1, \cdots, p_m$ in Definition \ref{orliczadditionmeasre-1} is the linear Orlicz addition, where $\varphi$ in formula (\ref{Orldef-function}) is replaced by \begin{equation*}
\varphi(x_1,\cdots,x_m)=\sum_{j=1}^m\alpha_j\varphi_j(x_j),
\end{equation*} with all  $\alpha_j> 0$, and with $\varphi_j$ either all in $\Phi_1$ or all in $\Psi_1$.  To obtain an interpretation for the $f$-divergence, we consider $m=2$,
$\alpha_1=1$, and $\alpha_2=\ee>0$. That is, let $\varphi(x_1, x_2)=\varphi_1(x_1)+\epsilon \varphi_2(x_2)$   for $\varphi_1,\varphi_2\in \Phi_1$ and let $p_1\widetilde{+}_{\varphi,\ee} p_2$ be given by
\begin{equation}\label{rlc-function}
\varphi_1\left(\frac{p_1(x)}{p_1\widetilde{+}_{\varphi,\ee} p_2(x)}\right)+\ee\varphi_2
\left(\frac{p_2(x)}{p_1\widetilde{+}_{\varphi,\ee} p_2(x)}\right)=1,
\end{equation} if $p_1(x)+p_2(x)>0$, and otherwise by 0.  We use the same formula for $p_1\widetilde{+}_{\varphi,\ee} p_2$ if $\varphi_1,\varphi_2\in \Psi_1$ and
$p_1, p_2\in \cM^+$.

\subsection{An interpretation of the $f$-divergence}
The following lemma is needed for our interpretation of the $f$-divergence. Denote by $(\varphi_1)'_l(1)$ and $(\varphi_1)'_r(1)$  the left and, respectively, the right derivatives of $\varphi_1$ at $t=1$  if they exist. Let $\Phi_1^{(1)}$ and $\Psi_1^{(1)}$ stand for the set of functions $\varphi\in \Phi_1$ and, respectively, $\varphi\in \Psi_1$, such that $\varphi(1)=1$.

\begin{thm}\label{maindov}
Let $\varphi_1, \varphi_2 \in {\Phi}_1^{(1)}$  such that $(\varphi_1)'_l(1)$ exists and is positive.  Let  $A\subset \Omega$ be measurable with $\mu(A)\neq 0$, and $p_1\in \cM^+\cap \mathcal{L}_{s, A}$ and $p_2\in \cM\cap  \mathcal{L}_{s, A}$   such that, $$\sup_{x\in A} \left(\frac{p_2(x)}{p_1(x)}\right) <a_1$$ for some constant $a_1<\infty$. Then, for  $0\neq s\in \R$, one has,
\begin{eqnarray}\label{ndov-s-norm}
&&(\varphi_1)'_l(1)\lim_{\ee\rightarrow 0^+}
\frac{\|p_1\widetilde{+}_{\varphi,\ee}p_2\|^s_{s, A}
-\|p_1\|^s_{s, A}}{s\cdot \ee}\nonumber\\
&&\ \  \ \ \ \ \ \ \ \ \ =
\int_{A}
\varphi_2\left(\frac{p_2(x)}{ p_1(x)}\right) \big[p_1(x)\big]^s\,d\mu(x).
\end{eqnarray}

If $\varphi_1, \varphi_2\in {\Psi}_1^{(1)}$ satisfy that  $(\varphi_1)'_r(1)$ exists and is nonzero, and  if $p_1, p_2\in \cM^+\cap  \mathcal{L}_{s, A}$ such that $$\inf_{x\in A}\bigg(\frac{p_2(x)}{p_1(x)}\bigg)>a_2$$ for some constant $a_2>0$,  then \eqref{ndov-s-norm} holds with $(\varphi_1)'_l(1)$ replaced by $(\varphi_1)'_r(1)$.
\end{thm}

\begin{proof}    Let $\varphi_1, \varphi_2 \in {\Phi}_1^{(1)}$ and $\ee\in (0,1]$.  As $\varphi_1$ and $\varphi_2$ are strictly increasing, one can easily check,  by an argument similar to the proof of Theorem \ref{dualOrthm1} (iii), that  for all $\ee\in (0, 1]$,
$$p_1(x) \leq p_1\widetilde{+}_{\varphi,\ee} p_2(x)
\leq p_1\widetilde{+}_{\varphi,1} p_2 (x),\ \  \mathrm{for\ all}\ x\in A. $$
This together with formula  (\ref{rlc-function}) yield, for all $x\in A$,
\begin{eqnarray*} 1&=& \varphi_1\bigg(\frac{p_1(x)}{p_1\widetilde{+}_{\varphi,\ee} p_2(x)}\bigg)+\ee\varphi_2
\bigg(\frac{p_2(x)}{p_1\widetilde{+}_{\varphi,\ee} p_2(x)}\bigg)\\  &\leq& \varphi_1\bigg(\frac{p_1(x)}{p_1\widetilde{+}_{\varphi,\ee} p_2(x)}\bigg)+\ee\varphi_2
\bigg(\frac{p_2(x)}{p_1(x)}\bigg)\\ &\leq&  \varphi_1\bigg(\frac{p_1(x)}{p_1\widetilde{+}_{\varphi,\ee} p_2(x)}\bigg)+\ee\varphi_2
(a_1),\end{eqnarray*} where we have used the assumption
$$ \sup_{x\in A}\bigg(\frac{p_2(x)}{p_1(x)}\bigg)<a_1<\infty.
$$ The above assumption also implies that  there is a constant $b_1<\infty$, s.t., for all $\ee\in (0, 1]$,  $$1 \leq \frac{p_1\widetilde{+}_{\varphi,\ee} p_2}{p_1}
\leq \frac{p_1\widetilde{+}_{\varphi,1} p_2}{p_1}<b_1\ \ \ \mathrm{on} \ A. $$

Let $\ee$ be small enough so that  $1-\ee \varphi_2(a_1)>0$. Then,
\begin{eqnarray*} \varphi_1^{-1}\big(1-\ee\varphi_2
(a_1) \big)\leq \frac{p_1(x)}{p_1\widetilde{+}_{\varphi,\ee} p_2(x)} ,\end{eqnarray*}
 and hence, for all $x\in A$,
\begin{eqnarray}\label{bound1}
0&\le& \frac{
p_1\widetilde{+}_{\varphi,\ee} p_2(x)- p_1(x)}{p_1(x)}\nonumber\\ &=& \bigg( \frac{
p_1\widetilde{+}_{\varphi,\ee} p_2(x)}{p_1(x) }\bigg)\cdot \bigg(1- \frac{p_1(x)}{p_1\widetilde{+}_{\varphi,\ee} p_2(x)}\bigg) \nonumber\\
&\le& b_1\cdot \left(1-\varphi_1^{-1}\left(1-\ee\varphi_2(a_1)\right)\right).
\end{eqnarray} 
Taking $\ee\rightarrow 0^+$, (\ref{bound1}) yields   \begin{equation}\label{uniformconvergent-0} \frac{p_1\widetilde{+}_{\varphi,\ee} p_2}{p_1}\rightarrow   1,  \ \mathrm{uniformly\  on}\  A \ \mathrm{as} \ \ee\rightarrow 0^+.  \end{equation}

For convenience,   let $$w(\ee, x)=\frac{p_1\widetilde{+}_{\varphi,\ee}p_2(x)-p_1(x)}{p_1(x)}.$$ Then $w(\ee, x)\rightarrow 0^+$ as $\ee \rightarrow 0^+$ by (\ref{uniformconvergent-0}).  For $x\in A$,  by  (\ref{rlc-function}) and (\ref{uniformconvergent-0}), \begin{eqnarray}  \lim_{\ee\rightarrow 0^+}\!\!\! \frac{w(\ee, x)}{\ee } \! \!\!\!&=&\!\!\!\! \lim_{\ee\rightarrow 0^+}\! \!\bigg(\!\frac{p_1\widetilde{+}_{\varphi,\ee} p_2(x)}{p_1(x)}\!\bigg)\!\cdot\! \lim_{\ee\rightarrow 0^+}\!\!  \frac{\left(\!1\!-\!\frac{p_1(x)}{p_1\widetilde{+}_{\varphi,\ee} p_2(x)}\!\right)}{\ee}  \nonumber  \\ \!\!\!\!&= &\!\!\!\!\! \lim_{ \ee\rightarrow 0^+} \!\!\left(\!\frac{1\!-\!z(\ee)}{1\!-\!\varphi_1(z(\ee))} \!\!\right)\!\cdot\! \lim_{\ee\rightarrow 0^+} \!  \varphi_2\!
\left(\!\frac{p_2(x)}{p_1\widetilde{+}_{\varphi,\ee} p_2(x)}\!\!\right)   \nonumber \\ \! \!\!\!&=&\!\!\!\! \frac{1}{(\varphi_1)'_l(1)}\cdot  \varphi_2\left(\!\frac{p_2(x)}{ p_1(x)}\!\right) \label{limit-1-1-1}
\end{eqnarray} where we have used  $$z(\ee)=\varphi_1^{-1} \left(1-\ee\varphi_2
\left(\frac{p_2(x)}{p_1\widetilde{+}_{\varphi,\ee} p_2(x)}\right)\right)\rightarrow 1^{-} $$  as $\ee\rightarrow 0^+$ (note that $\varphi_1\in \Phi_1^{(1)}$ is increasing). This further implies that for all $x\in A$,
 \begin{eqnarray} 0 &\leq& \lim_{\ee\rightarrow 0^+}
\frac{\big[p_1\widetilde{+}_{\varphi,\ee}p_2(x)\big]^s
-\big[ p_1(x)\big]^s}{s\cdot \ee}\nonumber\\  &=& \lim_{\ee\rightarrow 0^+}
\frac{\big[1+ w(\ee, x)\big]^s
-1}{s\cdot \ee}\cdot \big[p_1(x)\big]^s \nonumber \\ &=&   \lim_{\ee\rightarrow 0^+}
\frac{\big[1+ w(\ee, x)\big]^s
-1}{s\cdot w(\ee, x)}\cdot \lim_{\ee\rightarrow 0^+}  \frac{w(\ee, x)}{ \ee}\cdot \big[p_1(x)\big]^s\nonumber\\ &=& \frac{1}{(\varphi_1)'_l(1)}\cdot \varphi_2\left(\frac{p_2(x)}{ p_1(x)}\!\right)  \cdot \big[p_1(x)\big]^s. \label{unform-bound-4-1} \end{eqnarray}

Moreover, by inequality (\ref{bound1}) and a calculation similar to (\ref{unform-bound-4-1}), we get, for $\ee<1/\varphi_2(a_1)$, for $0\neq s\in \R$ and for all $x\in A$, \begin{eqnarray*} 0&\leq&
\lim_{\ee\rightarrow 0^+}
\frac{\big[p_1\widetilde{+}_{\varphi,\ee}p_2(x)\big]^s
-\big[ p_1(x)\big]^s}{s\cdot \ee}\nonumber\\  &\leq&  \lim_{\ee\rightarrow 0^+}
\frac{\big[1+ u(\ee)\big]^s
-1}{s\cdot \ee}\cdot \big[p_1(x)\big]^s \\  &=&\frac{b_1\cdot  \varphi_2\left(a_1\right)}{(\varphi_1)'_l(1)} \cdot \big[p_1(x)\big]^s, \end{eqnarray*} where $u(\ee)=b_1\cdot \left(1-\varphi_1^{-1}\left(1-\ee\varphi_2(a_1)\right)\right) $ and
  \begin{eqnarray*} \lim_{\ee\rightarrow 0^+} \!\! \frac{u(\ee)}{\ee}= \frac{b_1\cdot  \varphi_2\left(a_1\right)}{(\varphi_1)'_l(1)}\end{eqnarray*} follows from a calculation similar to (\ref{limit-1-1-1}).  Hence, for $0\neq s\in \R$, one can find $\ee_0<1/\varphi_2(a_1)$, such that, for all $0<\ee<\ee_0$ and for all $x\in A$,
   \begin{eqnarray*}
\frac{\big[p_1\widetilde{+}_{\varphi,\ee}p_2(x)\big]^s
-\big[ p_1(x)\big]^s}{s\cdot \ee}  \leq  \frac{2b_1\cdot  \varphi_2\left(a_1\right)}{(\varphi_1)'_l(1)} \cdot \big[p_1(x)\big]^s. 
\end{eqnarray*}

Note that $p_1\in \cM^+\cap \mathcal{L}_{s, A}$, hence  $$0\leq \int_A  \frac{2b_1\cdot  \varphi_2\left(a_1\right)}{(\varphi_1)'_l(1)} \cdot \big[p_1(x)\big]^s\,d\mu(x)<\infty.$$ The desired formula (\ref{ndov-s-norm}) then follows by the Lebesgue dominant convergent theorem. That is,    \begin{eqnarray*} &&(\varphi_1)'_l(1)\lim_{\ee\rightarrow 0^+}
\frac{\|p_1\widetilde{+}_{\varphi,\ee}p_2\|^s_{s, A}
-\|p_1\|^s_{s, A}}{s\cdot \ee}\nonumber\\  && \  \  = (\varphi_1)'_l(1) \int_A \lim_{\ee\rightarrow 0^+}
\frac{\big[1+ w(\ee, x)\big]^s
-1}{s\cdot \ee}\cdot \big[p_1(x)\big]^s\,d\mu(x) \\
&&\ \  =
\int_{A}
\varphi_2\left(\frac{p_2(x)}{ p_1(x)}\right) \big[p_1(x)\big]^s\,d\mu(x).
\end{eqnarray*}

The case for $\varphi_1, \varphi_2 \in {\Psi}_1^{(1)}$ can be proved along the same lines.  For completeness, we include a brief proof with modification emphasized. Assume that  $$\inf_{x\in A}\bigg(\frac{p_2(x)}{p_1(x)}\bigg)>a_2>0.
$$ If  $\varphi_1, \varphi_2 \in {\Psi}_1^{(1)}$ and $\ee\in (0,1]$, then for $x\in A$,  $$  \frac{p_1\widetilde{+}_{\varphi, 1}p_2(x)}{p_1(x)} \leq \frac{p_1\widetilde{+}_{\varphi,\ee}p_2(x)}{p_1(x)} \leq 1,$$ where $p_1, p_2\in \cM^+$. Note that  $\varphi_1$ and $\varphi_2$ are decreasing. Hence, formula (\ref{rlc-function}) yields, for all $x\in A$,
\begin{eqnarray*} 1 \leq   \varphi_1\bigg(\frac{p_1(x)}{p_1\widetilde{+}_{\varphi,\ee} p_2(x)}\bigg)+\ee\varphi_2
(a_2).\end{eqnarray*}
Let $\ee<1/ \varphi_2(a_2).$ Similar to inequality (\ref{bound1}), one has,  \begin{eqnarray*}
0  &\leq&  \frac{ p_1(x)-
p_1\widetilde{+}_{\varphi,\ee} p_2(x)}{p_1(x)} \\ &=&   \bigg(\frac{p_1\widetilde{+}_{\varphi,\ee} p_2(x)}{p_1(x)}\bigg) \cdot   \bigg(\frac{p_1(x)} {p_1\widetilde{+}_{\varphi,\ee} p_2(x)}-1\bigg) \\ &\leq&  \varphi_1^{-1}\left(1-\ee\varphi_2(a_2)\right)-1.
\end{eqnarray*} This yields  (\ref{uniformconvergent-0}) if we let $\ee\rightarrow 0^+$. Moreover,
\begin{eqnarray} \lim_{\ee\rightarrow 0^+}\! \frac{\varphi_1^{-1}\!\left(1-\ee\varphi_2(a_2)\right)-1}{\ee} \!\!\!\!&=&\!\!\!\! -\! \lim_{\ee\rightarrow 0^+} \! \frac{\bar{z}(\ee)-1}{\varphi_1(\bar{z}(\ee))-1} \cdot \varphi_2(a_2) \nonumber \\ \!\!\!\!&=&\!\!\!\! -\frac{\varphi_2(a_2)}{(\varphi_1)'_r(1)},\label{limit-constant-1} \end{eqnarray} because $\bar{z}(\ee)=\varphi_1^{-1}\left(1-\ee\varphi_2(a_2)\right)\rightarrow 1^+$ as $\ee\rightarrow 0^+$ (note that $\varphi_1$ is decreasing). Hence,  one can find $\ee_0<1/\varphi_2(a_2)$, such that for all $0<\ee<\ee_0$ and for all $x\in A$,
   \begin{eqnarray*}  \frac{ \big[p_1(x)\big]^s-
\big[p_1\widetilde{+}_{\varphi,\ee} p_2(x) \big]^s}{p_1(x)\cdot \ee} \leq - \frac{2  \varphi_2\left(a_1\right)}{(\varphi_1)'_r(1)}\cdot \big[p_s(x)\big]^s.
\end{eqnarray*}  Following the calculations for (\ref{unform-bound-4-1}) and (\ref{limit-constant-1}), one can get,  for all $x\in A$,  \begin{eqnarray*}0&\leq&   \lim_{\ee\rightarrow 0^+}
\frac{\big[ p_1(x)\big]^s-\big[p_1\widetilde{+}_{\varphi,\ee}p_2(x)\big]^s
}{s\cdot \ee}\\ &=& - \frac{1}{(\varphi_1)'_r(1)}\cdot \varphi_2\left(\frac{p_2(x)}{ p_1(x)}\!\right) \cdot \big[p_1(x)\big]^s.   \end{eqnarray*} The desired formula (\ref{ndov-s-norm}) then follows by the Lebesgue dominant convergent theorem.
\end{proof}

Let $p_1$ and $p_2$ be density functions of measures $P_1\in\cF^+ $ and $P_2\in \cF$  respectively. Consider $A=\Omega$ and $s=1$. Under the assumptions stated in Theorem \ref{maindov},  formula (\ref{ndov-s-norm}) becomes, if one notices the definition of the $f$-divergence given in  (\ref{2016-1-20-f-def}),
 \begin{eqnarray}
\lim_{\ee\rightarrow 0^+}
\frac{P_1\widetilde{+}_{\varphi,\ee}P_2 (\Omega)
\!-\! P_1(\Omega)}{\ee} = \frac{1}{ (\varphi_1)'_l(1)}\cdot D_{\varphi_2} (P_2, P_1), \label{f-div-rep--1} \end{eqnarray} where the measure $P_1\widetilde{+}_{\varphi,\ee}P_2$ refers to the measure with the density function $p_1\widetilde{+}_{\varphi,\ee} p_2$. In other words,  we provide an interpretation for the  $f$-divergence by the linear Orlicz addition of measures.

\subsection{The Orlicz mixed volume and its dual}
 Again, with suitable selections of $\Omega, \mu,  P_1, P_2$ etc, one can obtain many interesting and important results. 
 
 For a continuous function $\phi: [0, \infty)\rightarrow \R$, define $\widetilde{V}_{\phi}(K, L)$, the dual Orlicz mixed volume of star bodies $K$ and $L$,  by \begin{eqnarray*}
\widetilde{V}_{\phi}(K, L)= \frac{1}{n} \int_{S^{n-1}}
\phi\left(\frac{\rho_L(u)}{\rho_K(u)}\right)[\rho_K(u)]^n\,d\sigma(u). \end{eqnarray*}  The dual Orlicz mixed volume is a central concept in the dual Orlicz-Brunn-Minkowski theory. It can be obtained by formula (\ref{ndov-s-norm}), if we let  $\Omega=S^{n-1}$, $n\cdot \mu=\sigma$ the spherical measure on $S^{n-1}$, $s=n$,  $p_1=\rho_{K}$,  $p_2=\rho_{L}$ and the star body $K\widetilde{+}_{\varphi,\ee}L$ determined by, for $u\in S^{n-1}$, \begin{equation*}
\varphi_1\bigg(\frac{\rho_K(u)}{\rho_{K\widetilde{+}_{\varphi,\ee} L}(u)}\bigg)+\ee\varphi_2
\bigg(\frac{\rho_L(u)}{\rho_{K\widetilde{+}_{\varphi,\ee} L}(u)}\bigg)=1.
\end{equation*} That is, \begin{equation*}
(\varphi_1)'_l(1)\lim_{\ee\rightarrow 0^+}\frac{V_n(K\widetilde{+}_{\varphi,\ee}L)-V_n(K)}{n\cdot \ee}=\widetilde{V}_{\varphi_2}(K, L).
\end{equation*} Please see Theorem 5.4 in \cite{ghwy14} for more precise statements.

Now we prove the following theorem regarding the $L_s$ mixed volume given by (\ref{Ls-mixed-volume-2016}). Let $K, L$ be convex bodies with the origin in their interiors. Let $\Omega=S^{n-1}$ and $n\cdot \, d\mu=h_K^{1-s}\, dS_K$.  Define the convex body $K+_{\varphi, \ee}L$ by its support function  $h_{K+_{\varphi,\ee} L}$,  the unique solution of  \begin{equation*}
\varphi_1\bigg(\frac{h_K(u)}{h_{K+_{\varphi,\ee} L}(u)}\bigg)+\ee\varphi_2
\bigg(\frac{h_L(u)}{h_{K+_{\varphi,\ee} L}(u)}\bigg)=1,
\end{equation*} for $u\in S^{n-1}$ and  for convex functions $\varphi_1, \varphi_2\in \Phi_1^{(1)}$.

\begin{cor} Let $K, L$ be convex bodies with the origin in their interiors. Assume that convex functions $\varphi_1, \varphi_2\in \Phi_1^{(1)}$ satisfy the conditions in  Theorem \ref{maindov}. Then, for $0\neq s\in \R$, \begin{eqnarray*} &&(\varphi_1)'_l(1)\lim_{\ee\rightarrow 0^+}
\frac{V_s(K, K+_{\varphi,\ee}L)
-V_n(K)}{s\cdot \ee}= V_{\varphi_2}(K, L), \end{eqnarray*} where $V_{\phi}(K, L)$ is the Orlicz $\phi$-mixed volume (\cite{ghw14,xjl14,y15})  defined by   \begin{eqnarray*}V_{\varphi_2}(K, L) =\frac{1}{n}
\int_{S^{n-1}}
\varphi_2\left(\frac{h_L(u)}{ h_K(u)}\right)  h_K(u) \, dS_K(u).
\end{eqnarray*}\end{cor}

\begin{proof} Let  $\Omega=S^{n-1}$ and $n\cdot \, d\mu=h_K^{1-s}\, dS_K$.   Let $p_1 =h_K$ and $p_2=h_L$.  Note that if $K, L$ are convex bodies, then $p_1$ and $p_2$ satisfy the assumptions in Theorem \ref{maindov} automatically.  The corollary follows immediately from Theorem \ref{maindov} and the fact  $V_s(K, L)=\|h_L(u)\|^s_{s, S^{n-1}}.$      \end{proof} In other words, we provide a new  interpretation for the Orlicz $\phi$-mixed volume, which is different from the one given by  \cite{ghw14,xjl14}: \begin{eqnarray*} &&(\varphi_1)'_l(1)\lim_{\ee\rightarrow 0^+}
\frac{V_n(K+_{\varphi,\ee}L)
-V_n(K)}{n \cdot \ee}=V_{\varphi_2}(K, L).\end{eqnarray*} It is worth to mention that the Orlicz $\phi$-mixed volume is a fundamental object in the Orlicz-Brunn-Minkowski theory for convex bodies; and it plays important roles in, e.g., the Orlicz-Minkowski inequality \cite{ghw14,xjl14}, and the Orlicz affine and geominimal surface areas \cite{y15}.

 \section{An inequality equivalent to Jensen's inequality}

With the linear Orlicz addition of functions, we can prove that the classical Jensen's inequality has an equivalent form. For $\alpha_1, \alpha_2> 0$, let  \begin{equation}\label{phrest}
\varphi(x_1, x_2)= \alpha_1\varphi_1(x_1)+ \alpha_2\varphi_2(x_2),
\end{equation} with $\varphi_1, \varphi_2$ are either both in $\Phi_1$ or both in $\Psi_1$. For this special $\varphi$, the dual functional Orlicz-Brunn-Minkowski inequality in Theorem \ref{dualOBMI1} can be rewritten as:  \begin{equation} \label{dualOBMI22-22} \alpha_1 \varphi_1\!\bigg(\!\frac{P_1(\Omega)}
{\widetilde{+}_{\varphi}(P_1, P_2)(\Omega)}\!\bigg)\!+\alpha_2\varphi_2\!\bigg(\!\frac{P_2(\Omega)}
{\widetilde{+}_{\varphi}(P_1,  P_2)(\Omega)} \! \bigg)\ge 1
\end{equation}    if $\varphi_1, \varphi_2$ are concave; and the direction of the inequality is reversed if $\varphi_1, \varphi_2$ are convex.  On the other hand, by Jensen's inequality, one can obtain the following inequality:  \begin{eqnarray}D_{\phi}(P_2, P_1)&=& \int_{\Omega}  \phi\bigg( \frac{p_2(x)}{p_1(x)} \bigg)p_1(x)\,d\mu(x) \nonumber \\ & \leq & P_1(\Omega) \cdot  \phi \left( \frac{P_2(\Omega)}{P_1(\Omega)}\right), \label{Iso:type:22-22} \end{eqnarray} if $\phi$ is concave; the direction of the inequality is reversed if $\phi$ is  convex. 
  If $\phi$ is strictly concave or convex and $p_1, p_2\in \cM^{+c}$, equality holds if and only if $p_2/p_1$ is a constant on $\Omega$. Note that H\"{o}lder's and Jensen's inequalities are special cases of  inequality (\ref{Iso:type:22-22}).

 \begin{thm}\label{crdm} Let $p_1, p_2, \varphi_1, \varphi_2$ satisfy the conditions in Theorem \ref{maindov}.  The dual functional Orlicz-Brunn-Minkowski inequality (\ref{dualOBMI22-22}) is equivalent to inequality (\ref{Iso:type:22-22}) in the following sense: if one of them holds, the other one also holds.

 Moreover, if the convexity or concavity of functions involved is strict and $p_1, p_2\in \cM^{+c}$, these two inequalities have the same characterization for equality.
  \end{thm}

\begin{proof} We only prove the case when  $\varphi_1, \varphi_2\in \Phi_1^{(1)}$  are concave. The proofs for other cases can be proved along the same lines.

Let $\varphi$ be as in (\ref{phrest}) for some constants $\alpha_1, \alpha_2>0$.  First, recall that $p_1\in \cM^+\cap \mathcal{L}_{s, \Omega}$. Statements (ii)-(iii) of Theorem \ref{dualOrthm1} yield $$0<P_1\widetilde{+}_{\varphi} P_2(\Omega) <\infty, $$ where $P_1\widetilde{+}_{\varphi} P_2$ is the measure with density function $p_1\widetilde{+}_{\varphi} p_2$ given by, for $x\in \Omega$, 
\begin{equation} \label{liner-11}
\alpha_1\varphi_1\left(\frac{p_1(x)}{p_1\widetilde{+}_{\varphi} p_2(x)}\right)+\alpha_2\varphi_2
\left(\frac{p_2(x)}{p_1\widetilde{+}_{\varphi} p_2(x)}\right)=1. \end{equation}

Suppose that inequality (\ref{Iso:type:22-22}) holds true. For the concave functions $\varphi_1, \varphi_2$, 
\begin{eqnarray*}\frac{D_{\varphi_2}(P_2, P_1\widetilde{+}_{\varphi} P_2)} {P_1\widetilde{+}_{\varphi} P_2 (\Omega)}  &\leq&   \varphi_2 \left( \frac{P_2(\Omega)}{P_1\widetilde{+}_{\varphi} P_2 (\Omega)}\right), \\ \frac{D_{\varphi_1}(P_1, P_1\widetilde{+}_{\varphi} P_2)} {P_1\widetilde{+}_{\varphi} P_2 (\Omega)}  &\leq&   \varphi_1 \left( \frac{P_1(\Omega)}{P_1\widetilde{+}_{\varphi} P_2 (\Omega)}\right). \end{eqnarray*}
It can be checked by (\ref{liner-11}) that
\begin{eqnarray*}1
\!\!\!&=&\!\!\!\alpha_1 \frac{D_{\varphi_1}(P_1, P_1\widetilde{+}_{\varphi} P_2)} {P_1\widetilde{+}_{\varphi} P_2 (\Omega)} +\alpha_2\frac{D_{\varphi_2}(P_2, P_1\widetilde{+}_{\varphi} P_2)} {P_1\widetilde{+}_{\varphi} P_2 (\Omega)}  \\ \!\!\!&\leq & \!\!\!
 \alpha_1 \varphi_1\!\bigg(\!\frac{P_1(\Omega)}
{\widetilde{+}_{\varphi}(P_1, P_2)(\Omega)}\!\bigg)\!+\alpha_2\varphi_2\!\bigg(\!\frac{P_2(\Omega)}
{\widetilde{+}_{\varphi}(P_1,  P_2)(\Omega)} \! \bigg).
\end{eqnarray*} That is the desired inequality (\ref{dualOBMI22-22}) holds.

On the other hand, assume that inequality (\ref{dualOBMI22-22}) holds for all $\alpha_1, \alpha_2>0$, in particular for $\alpha_1=1$ and $\alpha_2=\ee$. Then,
\begin{eqnarray*}
  \varphi_1\left(\!\frac{P_1(\Omega)}
{P_1\widetilde{+}_{\varphi,\ee}P_2(\Omega)}\!\right)
+\varepsilon \varphi_2\left(\!\frac{P_2(\Omega)}
{P_1\widetilde{+}_{\varphi,\ee}P_2(\Omega)}\!\right)
\geq 1
\end{eqnarray*}  which is equivalent to, for $\ee$ small enough,
\begin{eqnarray*}
 \!\frac{P_1(\Omega)}
{P_1\widetilde{+}_{\varphi,\ee}P_2(\Omega)} \geq   \varphi_1^{-1} \left(1-\varepsilon \varphi_2\left(\!\frac{P_2(\Omega)}
{P_1\widetilde{+}_{\varphi,\ee}P_2(\Omega)}\!\right)\right). \end{eqnarray*}  Together with (\ref{f-div-rep--1}), one gets,
 \begin{eqnarray*} \frac{D_{\varphi_2} (P_2, P_1)}{P_1(\Omega)}\!\!\!&=& \!\!\! \!(\varphi_1)'_l(1)
\lim_{\ee\rightarrow 0^+}
\frac{P_1\widetilde{+}_{\varphi,\ee}P_2 (\Omega)
\!-\! P_1(\Omega)}{\ee \cdot P_1(\Omega)} \\  \!\!\!&=&\!\!\! \! (\varphi_1)'_l(1)
\lim_{\ee\rightarrow 0^+}
\frac{1-\big(\frac{P_1(\Omega)}{P_1\widetilde{+}_{\varphi,\ee}P_2 (\Omega)} \big) }{\ee } \\ \!\!\!&\leq& \!\!\!\! (\varphi_1)'_l(1)\!
\lim_{\ee\rightarrow 0^+}\!\!\!
\frac{1\!-\!\varphi_1^{-1}\!\! \left(\!1\!-\!\varepsilon \varphi_2\!\!\left(\!\frac{P_2(\Omega)}
{P_1\widetilde{+}_{\varphi,\ee}P_2(\Omega)}\!\right)\!\right)}{\ee} \\ \!\!\! &=&\!\!\!\! \varphi_2\bigg(\frac{P_2(\Omega)}{P_1(\Omega)}\bigg)  \end{eqnarray*}  where the limit in the last equality can be obtained by a calculation similar to (\ref{limit-1-1-1}). Hence, inequality (\ref{Iso:type:22-22}) holds.

Note that  if the functions involved are strict concave and $p_1, p_2\in \cM^{+c}$, these two inequalities have the same characterization for equality; that is, there is a constant $\alpha>0$ such that $p_1=\alpha p_2$ on $\Omega$.  \end{proof}

\section{An optimization problem  for the $f$-divergence and related affine isoperimetric inequalities}\label{doaff}

A general optimization problem for the Csisz\'{a}r's  $f$-divergence can be described as follows: for a fixed measure $P_1\!\in\! \cF$ and a set of measures $\E\!\subset\! \cF$,  find  \begin{equation}\label{optimization-1=1=1} \inf_{P_2\in\E}D_f(P_2,P_1)\ \ \ \ \mathrm{or} \ \ \ \ \sup_{P_2\in\E}D_f(P_2,P_1),\end{equation} where the infimum and supremum depend on the convexity and concavity of $f$. The optimization problem (\ref{optimization-1=1=1}) contains many important objects in the information theory as special cases, such as the famous $I$-divergence geometry of probability distributions (see e.g., the highly cited paper by Csisz\'{a}r  \cite{c75}).  
 
 In this section, we link the optimization problem (\ref{optimization-1=1=1}) with Orlicz affine and geominimal surface areas in geometry. Then, we propose a special optimization problem and establish related functional affine isoperimetric inequalities. 
\subsection{Connection between the optimization problem (\ref{optimization-1=1=1}) and Orlicz affine and geominimal surface areas} \label{optimal-geominimal}

 With appropriate selections of geometric measures on convex or star bodies,  the optimization problem (\ref{optimization-1=1=1}) leads to fundamental geometric notions, for instance, the dual Orlicz affine and geominimal surface areas \cite{y14a}. Let $\phi: (0, \infty)\rightarrow (0, \infty)$ such that $\phi_n(t)=\phi(t^{1/n})$ for all $t\in (0, \infty)$ is decreasing and strictly convex. The dual Orlicz geominimal surface area of a star body $K$ is defined by \begin{eqnarray*}
 \widetilde{G}_{\phi}^{orlicz} (K)=\inf _{L\in \mathcal{K}} \left\{n\widetilde{V}_{\phi}(K, L ) \right\}  \end{eqnarray*}  where $\mathcal{K}$ is the set of convex bodies with the following properties: if $L\in \mathcal{K}$, then $L$ is a convex body with its centroid at $o$ and  with $V_n(L^\circ)=V_n(B^n_2)$. Here, $B^n_2$ is the unit Euclidean ball of $\mathbb R^n$ and $L^\circ$ is the polar body of $L$ defined by $$L^\circ=\{y\in \R^n: \langle x, y\rangle \leq 1  \ \  \  \mathrm{for\ all} \ x\in L\}.$$  Translating to the language of the $f$-divergence, one can let $\Omega=S^{n-1}$, $n \cdot \mu=\sigma$ the spherical measure on $S^{n-1}$,  $\,dP=\rho_K^n\,d\mu$ and $\,dQ=\rho_L^n\,d\mu$. Then, $$\widetilde{G}_{\phi}^{orlicz} (K)=\inf_{Q\in \E} D_{\phi_n} (Q, P), $$ where  $\E$ contains all measures $\,dQ=\rho_L^n\,d\mu$ with $L\in \mathcal{K}$.  

An arguably more important concept is the Orlicz geominimal surface area for convex bodies, which can be defined by, if $\phi(t^{-1/n})$ is strictly convex on $t\in (0, \infty)$,  $$ G_{\phi}^{orlicz} (K)=\inf_{L\in \mathcal{K}} \bigg(\int_{S^{n-1}}
\phi\left(\frac{h_L(u)}{ h_K(u)}\right)  h_K(u) \, dS_K(u)\bigg).$$ Translating to the language of the $f$-divergence, one can let $\Omega=S^{n-1}$, $n\cdot \mu=S_K$ the surface area measure of $K$ on $S^{n-1}$,  $\,dP=h_K\,d\mu$, $\,dQ=h_L\,d\mu$, and $\E$ be the set containing all measures $\,dQ=h_L\,d\mu$ with $L\in \mathcal{K}$.  Then,   $$G_{\phi}^{orlicz} (K)=\inf_{Q\in \E} D_{\phi} (Q, P).$$ 

\subsection{Functional affine isoperimetric inequalities}

Motivated by the connection  between the optimization problem (\ref{optimization-1=1=1}) and Orlicz affine and geominimal surface areas, we propose the dual functional affine and geominimal surface areas for functions and/or measures.  To simplify our arguments, we make the following assumptions (and more general results could be established by slight modifications). Let  $\Omega = \mathbb R^n$, $\mu$ be the Lebesgue measure on $\mathbb R^n$, and $\gamma_n$ be the Gaussian function. That is,  $\gamma_n(x)=e^{-\frac{\|x\|_2^2}{2}}$ for $x\in \mathbb R^n$ where $\|\cdot\|_2$ denotes the usual Euclidean norm on $\R^n$.  

For $p\in \cM^+$, define $p_{x_0}^\circ: \mathbb R^n\rightarrow [0, \infty]$, the polar dual function of $p$ with respect to $x_0\in \mathbb R^n$,  by $$ p_{x_0}^\circ(y) =\inf _{x\in \mathbb R^n} \bigg( \frac{e^{-\langle x, y\rangle}}{p(x-x_0)}\bigg).$$ In particular, the polar dual function of $p\in \cM^+$ (with respect to $o$) is $$p^\circ(y) =\inf _{x\in \mathbb R^n}\bigg(\frac{e^{-\langle x, y\rangle}}{p(x)}\bigg). $$ Note that $\gamma_n^\circ =\gamma_n$ and hence $\gamma_n$ can be viewed as the ``unit Euclidean ball" of functions (in terms of the polar dual for functions). Consequently,  the Gaussian function $\gamma_n$ serves as the optimizers of many optimization problems in, such as, probability theory and information theory.  

Let  $\cD\subset \cM^+$ be the set given by $$\cD=\left\{p\in \cM^+:  \mu(p) \mu(p^\circ) \leq \big[\mu(\gamma_n)\big]^2\right\}, $$  where for simplicity,  $$\mu(p)=\int_{\mathbb R^n} p(x)\,dx.$$  Clearly, $\cD\neq \emptyset$ as $\gamma_n\in \cD$. Note that the choice of the set $\cD$ is not ad-hoc; it comes from the geometry of log-concave functions. In fact,  the functional Blaschke-Santal\'o inequality for log-concave functions (see e.g., \cite{akm04, fm07, Lehec2009b})  states that for a log-concave function $p$ (where $p$ can be written as $p=e^{-\psi}$ with $\psi$ a convex function),   there exists $z_0\in\R^n$ (indeed $z_0$ can be assumed to be the center of mass of $p$) such that 
\begin{equation}\mu(p) \mu(p_{z_0}^\circ) \leq \big[\mu(\gamma_n)\big]^2=(2\pi)^n. \label{function--1-ss-1}\end{equation}  Denote by $\mathcal{L}_c$ the set of all log-concave functions; and clearly all log-concave functions with barycenters at $o$ are in $\cD.$   

Let  $\phi: (0, \infty)\rightarrow (0, \infty)$ be either in $\Phi$ or in $\Psi$ with   \begin{eqnarray*} \Phi\!\!\!&=&\!\!\!\{\phi: \phi \ \mbox{is decreasing and strictly convex  on} \ (0, \infty) \}; \\ \Psi \!\!\!&=&\!\!\!\{\phi: \phi \ \mbox{ is increasing and strictly concave on} \ (0, \infty) \}.\end{eqnarray*}  When we say a measure $Q\in \cD$, we mean that $Q$ is a measure whose density function $q$ is in $\cD$. 

Now, we define  the dual functional Orlicz affine and geominimal surface areas of functions and/or measures.  Write by $q$ the density function of $Q\in \cF$.  

  \begin{deff} \label{dOai} For fixed measure $P\in \cF^+$, the dual functional Orlicz affine surface area of $P$ is defined by \begin{eqnarray}\label{dOai1}
	\Orliczdual(P)=
	\inf_{	Q\in \cD}
	D_{\phi}\bigg(\frac{\mu(q^\circ)}{\mu(\gamma_n)} Q, P\bigg)
	\end{eqnarray}  for $\phi \in \Phi$; while for $\phi \in \Psi$, $\Orliczdual(P)$ is defined similarly but with ``$\inf$" replaced by ``$\sup$". 
	
In a similar way, with $\cD$ replaced by $\cD\cap \mathcal{L}_c$, we can define $\OrliczdualG(P)$,  the dual functional Orlicz geominimal surface area of $P$. 	
 \end{deff}
 
It can be easily checked that if $\phi$ is a constant $\alpha>0$, then  $\Orliczdual(P)=\OrliczdualG(P)=\alpha P(\mathbb R^n)$ for any fixed measure $P\in \cF^{+}$. It is also clear that $$\Orliczdual(P)\leq \OrliczdualG(P)$$ if $\phi\in \Phi$; while if $\phi\in \Psi$, $\Orliczdual(P)\geq \OrliczdualG(P)$.  

In general, it is not easy to calculate $\Orliczdual(P)$ and $\OrliczdualG(P)$, except when $P$ is a Gaussian measure.  To this end,  for $c>0$ a constant,  let $(\gamma_n\circ c) (x)=\gamma_n(cx)$ for all $x\in \mathbb R^n$.  Note that $(\gamma_n\circ c)^\circ=\gamma_n\circ c^{-1}$. By letting $q=\gamma_n\circ c$ which belongs to $\mathcal D$, one has,  \begin{eqnarray*} 
\Orliczdual(\gamma_n\circ c)\!\! &=&\!\! \inf_{Q\in \cD}
	D_{\phi}\bigg(\frac{\mu(q^\circ)}{\mu(\gamma_n)} Q, \gamma_n\circ c\bigg)\\ \!\!&\leq& \!\! \phi(c^n) \cdot \int_{\mathbb R^n} e^{-\frac{\|cx\|_2^2}{2}}\,dx. 
\end{eqnarray*}  On the other hand, as $\phi\in \Phi$ is convex, Jensen's inequality implies that \begin{eqnarray*}  \Orliczdual(\gamma_n\circ c) &\geq& \inf_{Q\in \cD} \mu (\gamma _n \circ c)  \cdot \phi\bigg(\frac{\mu(q) \mu(q^\circ) }{\mu(\gamma_n) \mu(\gamma_n\circ c)}\bigg)    \\ &\geq&  \inf_{Q\in \cD}  \phi(c^n) \cdot \int_{\mathbb R^n} e^{-\frac{\|cx\|_2^2}{2}}\,dx \\ &=&  \phi(c^n) \cdot \int_{\mathbb R^n} e^{-\frac{\|cx\|_2^2}{2}}\,dx, \end{eqnarray*} where the second inequality follows from the definition of $\mathcal{D}$ and the fact that $\phi \in \Phi$ is decreasing. That is, if $\phi\in \Phi$, then  \begin{eqnarray} 
\Orliczdual(\gamma_n\circ c) &=& \phi(c^n) \cdot \int_{\mathbb R^n} e^{-\frac{\|cx\|_2^2}{2}}\,dx\nonumber \\ &=&  \Big(\frac{\sqrt{2\pi}}{c}\Big)^n \cdot \phi(c^n). \label{equlity-gaussion} 
\end{eqnarray} This result also holds for  $\phi\in \Psi$. Moreover, if $\phi\in \Phi\cup \Psi$,   
\begin{eqnarray*} 
\OrliczdualG(\gamma_n\circ c) =\Big(\frac{\sqrt{2\pi}}{c}\Big)^n \cdot \phi(c^n). 
\end{eqnarray*} 

 Let $T$ be a linear transform on $\mathbb R^n$ with determinant $\pm 1$.  First of all, for all $p\in \cM^+$,  \begin{eqnarray*} (p\circ T)^\circ(y) &=& \inf _{x\in \mathbb R^n} \bigg(\frac{e^{-\langle x, y\rangle}}{(p\circ T)(x)}\bigg) \\ &=& \inf _{z\in \mathbb R^n} \bigg(\frac{e^{-\langle T^{-1} z, y\rangle}}{p(z)}\bigg)\\ &=&p^\circ (T^{-t}y),  \end{eqnarray*}  where $T^{-1}$ denotes the inverse of $T$ and $T^{-t}$ the transpose of $T^{-1}$. An easy argument by the substitution $z=Tx$ yields  $$\mu(p\circ T)=\int_{\mathbb R^n} (p\circ T)(x)\,dx =\int_{\mathbb R^n} p(z)\,dz. $$ Similarly,  $\mu\big((p\circ T)^\circ\big)=\mu(p^\circ)$ and hence $p\circ T\in \cD$ if $p\in \cD$. 

On the other hand, we can check that \begin{eqnarray*}  D_{\phi}(Q\circ T, P\circ T)\!\!\!&=&\!\!\! \int _{\mathbb R^n} \! \phi\left(\frac{(q\circ T)(x)}{(p\circ T)(x)}\right) (p\circ T)(x)\,dx\\ \!\!\! &=&  \!\!\!  \int _{\mathbb R^n} \! \phi\left(\frac{q(z)}{p(z)}\right) p(z)\,dz\\ &=&D_{\phi}(Q, P).  \end{eqnarray*}  Taking the infimum if $\phi \in \Phi$  (or supremum if $\phi\in \Psi$) over $\mathcal{D}$, one gets  $$\Orliczdual(P\circ T)= \Orliczdual(P).$$ In fact, we have proved the following result, which asserts that both $\Orliczdual(\cdot)$ and $\OrliczdualG(\cdot)$ are invariant under the volume preserving (invertible) linear transforms. 

\begin{thm} Let $T$ be a linear transform on $\mathbb R^n$ with determinant to be $\pm 1$.  For any $P\in \cF^+$, one has, $$\Orliczdual(P\circ T)= \Orliczdual(P)$$ where $P\circ T\in \cF^+$ is the measure with density function $p\circ T(x)=p(Tx)$ for all $x\in \mathbb R^n$; and $$\OrliczdualG(P\circ T)= \OrliczdualG(P).$$ \end{thm}  

The functional affine isoperimetric inequality aims to provide upper and/or lower bounds for an affine invariant functional defined on functions. Here, an affine invariant functional $\mathcal{G}: \cM^+\rightarrow \mathbb R$  is a functional such that $$\mathcal{G}(p)=\mathcal{G}(p\circ T)$$ for all $p\in \cM^+$ and for all invertible linear transform $T$ on $\mathbb R^n$ with determinant $\pm 1$. For example, $\mu(p)\mu(p^\circ)$ is  an affine invariant functional, and the celebrated functional Blaschke-Santal\'{o} inequality (\ref{function--1-ss-1}) is  a typical example of the functional affine isoperimetric inequality.

 Another example of such affine invariant functionals is $$\mathcal{G}(p)=\Orliczdual(P). $$ The following functional affine isoperimetric inequality provides upper and/or lower bounds for $\Orliczdual(P)$.  
\begin{thm}\label{isoine} For $\phi \in \Phi$, one has, 
 \begin{eqnarray*}\OrliczdualG(P)\geq  \Orliczdual(P)\geq \Orliczdual (\gamma_n\circ c)=\OrliczdualG (\gamma_n\circ c), 
 \end{eqnarray*}  where $c>0$ is the constant determined by $$c=\left(\frac{\mu(\gamma_n)}{\mu(p)}\right)^{1/n}.$$ The inequalities hold for $\phi \in \Psi$ with `` $\geq$" replaced by `` $\leq$". 
\end{thm}

\begin{proof} Note that the function $\phi\in \Phi$ is  decreasing and strictly convex. Jensen's inequality implies that  \begin{eqnarray*} \Orliczdual(P)&=&\inf_{Q\in \cD}
	D_{\phi}\bigg(\frac{\mu(q^\circ)}{\mu(\gamma_n)} Q, P\bigg)
 \\ &\geq& \mu(p)  \inf_{Q\in \cD} \phi\bigg(\frac{\mu(q) \mu(q^\circ) }{\mu(\gamma_n) \mu(p)}\bigg)\\ &=&    \mu(p)  \phi\bigg(\frac{\mu(\gamma_n)  }{\mu(p)}\bigg) \\ &=&  \phi(c^n)  c^{-n} \mu(\gamma_n)\\ &=& \Orliczdual(\gamma_n\circ c)  \end{eqnarray*}  where the second equality follows from the fact that $\phi$ is decreasing and $\mu(q)\mu(q^\circ)\leq \mu(\gamma_n)^2$, and  the last equality follows from formula (\ref{equlity-gaussion}).  
 
For $\phi\in \Psi$, which is  increasing and strictly concave, Jensen's inequality implies that  \begin{eqnarray*} \Orliczdual(P)&=&\sup_{Q\in \cD}
	D_{\phi}\bigg(\frac{\mu(q^\circ)}{\mu(\gamma_n)} Q, P\bigg)
 \\ &\leq& \mu(p)  \sup_{Q\in \cD} \phi\bigg(\frac{\mu(q) \mu(q^\circ) }{\mu(\gamma_n) \mu(p)}\bigg)\\ &=&    \mu(p) \phi\bigg(\frac{\mu(\gamma_n)  }{\mu(p)}\bigg) \\ &=& \Orliczdual(\gamma_n\circ c)  \end{eqnarray*}  where the second equality follows from the fact that $\phi$ is increasing and $\mu(q)\mu(q^\circ)\leq \mu(\gamma_n)^2$, and  the last equality follows from formula (\ref{equlity-gaussion}).  \end{proof} 
  
    Theorem \ref{isoine} states that, among all measures  $P\in \cF^+$,  the dual functional Orlicz affine and geominimal surface areas for $\phi\in \Phi$ attain their minimums at the Gaussian measures; while if $\phi\in \Psi$, their maximums are attained at the Gaussian measures.

The following functional affine isoperimetric inequality provides an upper bound for $\Orliczdual(P)$.  It states that, among all measures  $P\in \mathcal{D}$,  the dual functional Orlicz affine surface area for $\phi\in \Phi$ attain its maximum at the Gaussian measures.

 \begin{thm}\label{reverse-111-111--11} For measures $P\in \mathcal{D}$ and for $\phi\in \Phi$, one has,  \begin{eqnarray*} \Orliczdual(P)\leq \Orliczdual (\gamma_n\circ c_1), 
 \end{eqnarray*}  where $c_1>0$ is the constant determined by $$c_1=\left(\frac{\mu(p^\circ)}{\mu(\gamma_n)}\right)^{1/n}.$$   
 \end{thm} 
\begin{proof} Let  $\phi\in \Phi$. By (\ref{dOai1}) and $P\in \mathcal{D}$,  one has,   \begin{eqnarray*} \Orliczdual(P)&=&\inf_{Q\in \cD}
	D_{\phi}\bigg(\frac{\mu(q^\circ)}{\mu(\gamma_n)} Q, P\bigg)
 \\ &\leq& D_{\phi}\bigg(\frac{\mu(p^\circ)}{\mu(\gamma_n)} P, P\bigg)\\ &=&    \mu(p)  \phi\bigg(\frac{\mu(p^\circ)}{\mu(\gamma_n)  }\bigg) \\ &\leq &  \phi(c_1^n)  c_1^{-n} \mu(\gamma_n)\\ &=& \Orliczdual(\gamma_n\circ c_1)  \end{eqnarray*}  where the first inequality follows by letting $Q=P$, the second inequality follows from $\mu(q)\mu(q^\circ)\leq \mu(\gamma_n)^2$, and  the last equality follows from formula (\ref{equlity-gaussion}).   \end{proof}

Along the same lines, we can prove the following functional affine isoperimetric inequality for $\OrliczdualG(P)$.  It states that, among all log-concave measures  $P\in \mathcal{D}$,  the dual functional Orlicz geominimal surface area for $\phi\in \Phi$ attain its maximum at the Gaussian measures. \begin{thm} \label{reverse-111-111--22} Let $P\in \mathcal{D}$ be a log-concave measure whose density function $p\in \cM^+$ is a log-concave function.   Then, for $\phi\in \Phi$, one has,  \begin{eqnarray*} \Orliczdual(P)\leq \OrliczdualG(P) \leq \OrliczdualG (\gamma_n\circ c_1), 
 \end{eqnarray*}  where $c_1>0$ is the constant given in Theorem  \ref{reverse-111-111--11}.  \end{thm} 
 
When $\phi: (0, \infty)\rightarrow (0, \infty)$ is a strictly convex function but $\phi \notin \Phi$ (hence $\phi$ is not decreasing), one can still define the dual functional Orlicz affine surface area of $P\in \cF^+$ by
\begin{eqnarray*}
	\Orliczdual(P)=
	\inf_{	Q\in \cD}
	D_{\phi}\bigg(\frac{\mu(q^\circ)}{\mu(\gamma_n)} Q, P\bigg),
	\end{eqnarray*}  and the dual functional Orlicz geominimal surface area of $P\in \cF^+$ with  $\cD$ replaced by $\cD\cap \mathcal{L}_c$. These functionals are again affine invariant, but we are not able to calculate $\Orliczdual(\gamma_n\circ c)$ and $\OrliczdualG(\gamma_n\circ c)$ precisely. However, we are still able to prove the following functional affine isoperimetric inequalities, whose proofs follow along the same lines as those in Theorem  \ref{reverse-111-111--11}. 
	 \begin{thm} Let  $\phi: (0, \infty)\rightarrow (0, \infty)$ be a strictly convex function but $\phi\notin \Phi$. For measures $P\in \mathcal{D}$, one has,  \begin{eqnarray*} \Orliczdual(P) &\leq&   \mu(p)  \phi\bigg(\frac{\mu(p^\circ)}{\mu(\gamma_n)  }\bigg)\\ &\leq &  \phi(c_1^n)  c_1^{-n} \mu(\gamma_n),  \end{eqnarray*} where $c_1>0$ is the constant given in Theorem  \ref{reverse-111-111--11}. 
	 
	   These inequalities also hold for the dual functional Orlicz geominimal surface area if  in addition $P\in \mathcal{D}$ is a log-concave measure. 
 \end{thm}

  \section{Closing Remarks} 
  
  This paper provides a functional analogue of the recently initiated dual Orlicz-Brunn-Minkowski theory for star bodies \cite{ghwy14, zzx14}.  With the help of the newly introduced Orlicz addition for measures, we are able to establish the dual functional Orlicz-Brunn-Minkowski inequality. Moreover,  we gave an interpretation for the famous Csisz\'{a}r's $f$-divergence.  Their applications and connections with geometry are also discussed. In particular, we are able to prove that  the dual functional Orlicz-Brunn-Minkowski inequality is equivalent to Jensen's inequality for integrals. 
  
 This paper further boosts the already existing connections between geometry and information theory. As explained in Subsection \ref{optimal-geominimal}, by choosing special measures and special set $\E$, we are able to translate fundamental geometric concepts into an optimization problem for the $f$-divergence.  In particular, we define the dual functional Orlicz affine and geominimal surface areas for functions, and establish related functional affine isoperimetric inequalities. As expected, these functional affine invariants for measures attain their minimums (or maximums) at the Gaussian measures under certain conditions on $\phi$. These functional affine isoperimetric inequalities are usually more important in applications.  
  
 Last but not the least,  the newly defined dual functional Orlicz affine and geominimal surface areas can be viewed as ``dual" concepts to the (Orlicz) affine and geominimal surface areas for log-concave functions \cite{ArtKlarSchuWer, Caglar-6, CaglarWerner, CaglarYe}. The latter ones are fundamental concepts in a rapidly developing field: geometrization of log-concave functions.

\section*{Acknowledgment}

The research of SH is supported by CSC. 
The research of DY is supported by a NSERC grant.

\ifCLASSOPTIONcaptionsoff
\newpage
\fi



%

\vskip 8mm
\noindent Shaoxiong Hou:  Department of Mathematics and Statistics, Memorial University, St. John's, NL A1C 5S7, Canada. {\tt email: s.hou@mun.ca}
\vskip 2mm
\noindent Deping Ye:  Department of Mathematics and Statistics, Memorial University, St. John's, NL A1C 5S7, Canada. {\tt email: deping.ye@mun.ca}

\end{document}